\documentclass[reqno,final,10pt]{amsart}

\usepackage[l2tabu,orthodox]{nag}
\usepackage{latexsym}
\usepackage{fixltx2e}
\usepackage{mathrsfs}
\usepackage{amssymb}
\usepackage{eulervm}
\usepackage{euscript}
\usepackage{mathtools}
\DeclareMathAlphabet{\mathbbm}{U}{bbm}{m}{n}
\usepackage[all]{onlyamsmath}
\usepackage{ifthen}
\newcommand{\Ifempty}[3]{\ifthenelse{\equal{#1}{}}{#2}{#3}}
\newcommand{\Ifnempty}[2]{\Ifempty{#1}{}{#2}}
\usepackage{ifdraft}
\usepackage[natbib=true,hyperref=true,backref=true,maxcitenames=2,maxbibnames=10,backend=biber]{biblatex}

\makeatletter
\DeclareFieldFormat{postnote}{\begingroup\blx@mkpageprefix{section}[\@firstofone]{#1}}
\makeatother
\DeclareFieldFormat{pages}{{\mkpagetotal[bookpagination]{#1}}}

\bibliography{manual,mr,my}
\usepackage[breaklinks,bookmarks,final]{hyperref}
\usepackage{verbatim}
\usepackage[arrow,matrix]{xy}

\newcommand{\cat}[1]{\mathcal{#1}} 
\newcommand{\ftr}[1]{\mathscr{#1}} 
\newcommand{\nat}[1]{\mathbbm{#1}} 
\newcommand{\obj}[1]{\mathit{#1}} 
\newcommand{\sch}[1]{\mathbold{#1}} 
\newcommand{\fset}[1]{\mathfrak{#1}} 
\newcommand{\set}[1]{{#1}} 
\newcommand{\Free}[1]{\mathcal{F}(#1)} 
\newcommand{\Ab}[1]{\mathcal{A}(#1)} 
\newcommand{\Aff}[1]{\underline{#1}} 

\newcommand{\St}[2]{\{#1|#2\}}
\newcommand{\LK}[1][\kappa]{<\!\!#1}

\newcommand{\OO}{\mathcal{O}} 

\newcommand{\fF}{\ftr{F}}
\newcommand{\fW}{\ftr{W}}

\newcommand{\fG}{\ftr{G}}

\newcommand{\cC}{\cat{C}}
\newcommand{\cD}{\cat{D}}
\newcommand{\cI}{\cat{I}}
\newcommand{\cS}{\cat{S}}
\newcommand{\oX}{\obj{X}}
\newcommand{\oY}{\obj{Y}}
\newcommand{\oZ}{\obj{Z}}
\newcommand{\oW}{\obj{W}}
\newcommand{\oI}{\obj{I}}
\newcommand{\oJ}{\obj{J}}
\newcommand{\oK}{\obj{K}}

\newcommand{\sX}{\sch{X}}
\newcommand{\sS}{\sch{S}}
\newcommand{\sY}{\sch{Y}}
\newcommand{\sZ}{\sch{Z}}
\newcommand{\sW}{\sch{W}}
\newcommand{\sG}{\sch{G}}
\newcommand{\sH}{\sch{H}}
\newcommand{\sV}{\sch{V}}
\newcommand{\sU}{\sch{U}}
\newcommand{\sT}{\sch{T}}
\newcommand{\GL}{\sch{GL}}
\newcommand{\sGm}{\sch{G_m}}
\newcommand{\sGa}{\sch{G_a}}

\newcommand{\Res}[2]{#1\vert_{#2}}

\newcommand{\FY}{\fset{M}}
\newcommand{\FZ}{\fset{N}}
\newcommand{\qF}{\mathcal{F}}

\newcommand{\Sch}[1][]{\cat{S}ch\Ifnempty{#1}{/{#1}}}

\newcommand{\dd}{\partial}

\DeclareMathOperator{\spec}{spec}
\DeclareMathOperator{\Sym}{Sym}
\DeclareMathOperator{\Ehom}{Hom}
\DeclareMathOperator{\Eend}{End}

\newcommand{\eps}{\epsilon}
\newcommand{\iso}{\Isom}

\newcommand{\point}[1][]{\subsubsection{#1}}
\newcommand{\Lmod}[2][\Cc]{#2\!-\!#1}
\newcommand{\Rcoh}[1][E]{(#1)_{coh}}
\newcommand{\Lcoh}[1][E]{(#1)^{coh}}
\newcommand{\Flat}[2][\Cc]{#1_{(#2)}}
\newcommand{\Plng}[2][\Cc]{#1^{(#2)}}
\newcommand{\Op}[1]{{#1}^{op}}

\newcommand{\ZZ}{\mathbb{Z}}
\newcommand{\NN}{\mathbb{N}}
\newcommand{\FF}{\mathbb{F}}
\newcommand{\QQ}{\mathbb{Q}}

\renewcommand{\AA}{\mathbb{A}}

\newcommand{\DD}{\mathbb{D}}

\newcommand{\Cc}{\mathcal{C}}
\newcommand{\Dd}{\mathcal{D}}

\newcommand{\sg}{\sigma}
\newcommand{\w}{\omega}

\newcommand{\ul}[1]{\underline{#1}}
\newcommand{\Alg}[1]{\overline{#1}}

\newcommand{\ti}[1]{\tilde{#1}}
\newcommand{\ra}[1][]{\xrightarrow{#1}}
\newcommand{\mt}{\mapsto}
\newcommand{\Ten}{\otimes}

\providecommand{\1}{\mathbold{1}}
\renewcommand{\1}{\mathbold{1}}

\newcommand{\x}{\times}
\newcommand{\Isom}{\ra[\sim]}
\newcommand{\kk}{\Bbbk}

\newcommand{\Hom}{\ul{\Ehom}}
\newcommand{\Aut}{\ul{Aut}}

\newcommand{\Co}[1]{{{#1}^\vee}}

\newcommand{\Def}[1]{\emph{#1}}

\DeclareMathAlphabet{\mathpzc}{OT1}{pzc}{m}{it}

\renewcommand{\Vec}{\mathpzc{Vec}}
\newcommand{\Rep}{\mathpzc{Rep}}

\newcommand{\Pro}[2][]{\operatorname{Pro}_{#1}({#2})}
\newcommand{\Ind}[2][]{\operatorname{Ind}_{#1}({#2})}

\newcommand{\Lim}[2][]{\underset{#1}{\varprojlim}\,#2}

\newcommand{\Sec}[1]{\S\ref{#1}}
\newcommand{\rpt}[1]{\ref{#1}}

\newcommand{\newthm}[2]{\newtheorem{#1}[subsubsection]{#2}}
\newthm{theorem}{Theorem}
\newthm{claim}{Claim}
\newthm{cor}{Corollary}
\newthm{lemma}{Lemma}
\newthm{prop}{Proposition}
\theoremstyle{definition}
\newthm{defn}{Definition}
\theoremstyle{remark}
\newthm{rmk}{Remark}
\newthm{xmpl}{Example}
\newenvironment{remark}{\begin{rmk}}{\qed\end{rmk}}
\newenvironment{example}{\begin{xmpl}}{\qed\end{xmpl}}

\newcommand{\DefAlias}[2]{\expandafter\xdef\csname #1\endcsname{#2}}
\newcommand{\CiteAlias}[2]{\DefAlias{CITE#1}{#2}}
\renewcommand{\Cite}[2][]{\Ifempty{#1}{\citet{\csname CITE#2\endcsname}}{\citet[#1]{\csname CITE#2\endcsname}}\wlog{Cite :: #2}}

\CiteAlias{buium}{MR2166202}
\CiteAlias{cassidy2}{MR0409426}
\CiteAlias{cassidysinger}{MR2322329}
\CiteAlias{crystal}{MR0491705}
\CiteAlias{CWM}{MR1712872}
\CiteAlias{Deligne}{MR1106898}
\CiteAlias{hardouinsinger}{MR2425146}
\CiteAlias{LNM900II}{LNM900II}
\CiteAlias{matsumura}{MR1011461}
\CiteAlias{MS1}{MR2602031}
\CiteAlias{MS2}{ms2}
\CiteAlias{ovchinde}{MR2480859}
\CiteAlias{pinkfgs}{pinkfgs}
\CiteAlias{Saavedra}{MR0338002}
\CiteAlias{singer}{MR1960772}
\CiteAlias{singer2}{MR1480919}
\CiteAlias{tannakian}{tannakian}
\CiteAlias{waterhouse}{MR547117}
\CiteAlias{restrict}{MR2723571}
\CiteAlias{lurie}{lurie}
\CiteAlias{ovchin}{MR2426137}
\CiteAlias{ggo}{ggo}
\CiteAlias{ziegler}{MR1959321}
\CiteAlias{borger}{borger}
\CiteAlias{sga3}{MR0274458}

\newcommand{\Deligne}[1]{\Cite[#1]{Deligne}}
\newcommand{\DM}[1]{\Cite[#1]{LNM900II}}

\title{Tannakian formalism over fields with operators}
\author[M. Kamensky]{Moshe Kamensky}
\address{
  Department of Math \\
  University of Notre-Dame \\
  Notre-Dame, IN\\
  USA
}
\email{\url{mailto:mkamensky@nd.edu}}
\urladdr{\url{http://mkamensky.notlong.com}}
\thanks{The author is thankful for support by NSF grant No.~1001741}
\subjclass[2010]{Primary 20G05,12H99,20G42; Secondary 34M15}
\keywords{Tensor categories, Tannakian formalism, Hasse--Schmidt derivations, 
prolongation spaces}

\begin{document}
\begin{abstract}
  We develop a theory of tensor categories over a field endowed with abstract 
  operators. Our notion of a ``field with operators'', coming from work of 
  Moosa and Scanlon, includes the familiar cases of differential and 
  difference fields, Hasse--Schmidt derivations, and their combinations. We 
  develop a corresponding Tannakian formalism, describing the category of 
  representations of linear groups defined over such fields. The paper 
  extends the previously know (classical) algebraic and differential 
  algebraic Tannakian formalisms.
\end{abstract}
\maketitle

\section*{Introduction}
We study fields with operators (briefly described below, and more thoroughly 
in~\Sec{sec:ms}), and linear groups over such fields.  Given such a group 
\(\sG\) (as defined in~\Sec{sec:main}), our goal is to describe the category 
\(\Rep_\sG\) of finite dimensional representations of \(\sG\), in a manner 
similar to the classical Tannakian formalism. In addition to generalising the 
usual Tannakian formalism, this paper forms a natural generalisation and 
reformulation of the theory of differential Tannakian categories 
(\Cite{ovchin}), and especially of the definition of differential tensor 
categories in~\Cite[4]{tannakian}.

We mention that results this kind are expected to have applications to Galois 
theory of linear equations with various operators. The classical Galois 
theories of ordinary differential and difference linear equations (as 
explained in \Cite{singer} and \Cite{singer2}, respectively) may be 
approached via the classical Tannakian formalism (also in~\Deligne{9}).  More 
recently, there are the Galois theory of (linear) partial differential 
equations (initiated by \Cite{cassidysinger}) to which the differential 
Tannakian theory mentioned above was applied in~\Cite{ovchinde} (see 
also~\Cite{ggo}), as well as linear equations involving both derivatives and 
automorphisms (\Cite{hardouinsinger}), and other variants. It is hoped that 
the present paper will provide the tools to approach all these Galois 
theories in a uniform manner, from the Tannakian point of view (of course, 
the classical Tannakian theory has many more applications in different areas, 
and we hope that similar applications will be found for the generalised 
theory in this paper). We sketch the definition of the Galois group (in the 
case of commuting automorphism and Hasse-Schmidt derivations) 
in~\Sec{ssc:galois} below.

The main result of the paper, describing  the analogue of tensor categories, 
as well as the statement that shows the notion to be adequate, i.e., that it 
does axiomatise categories of representations, is in~\Sec{sec:main} 
(specifically, Definition~\ref{def:etensor} and Theorem~\ref{thm:main}). Both 
the definition and the statement are rather immediate once the fundamental 
ideas are developed, so we now turn to a brief overview of the ideas that 
appear in the first two sections.

Our notion of a ``field with operators'' comes from (a variant of) the 
formalism developed by~\citet{MR2602031,ms2}. This formalism includes at 
least the cases of differential fields (fields endowed with a derivation, or 
a vector field), difference fields (fields with an endomorphism), 
Hasse--Schmidt derivations, and their combinations. To explain the idea, 
consider a field \(\kk\) with an endomorphism. One could alternatively 
describe the situation by saying that we are given an action of the monoid 
\(\NN\) of natural numbers on \(\kk\). More generally, one could consider the 
action of a monoid \(M\) on \(\kk\). When \(M\) is infinite, it cannot be 
viewed as a scheme. However, as a set, it is the (filtered) union of finite 
sets, each of which can be viewed as a scheme. Furthermore, the monoid 
operation maps the product of two finite sets in the system into another such 
finite set. In other words, \(M\) is a monoid in the category of ind-finite 
schemes, and we are given an action of \(M\) on \(\spec(\kk)\).

Since any set is the filtered union of its finite subsets, the description 
above accounts for all discrete monoid actions. However, some finite schemes 
do not come from finite sets. Recall that the data of a derivation on the 
field \(\kk\) over the subfield \(\kk_0\) is equivalent to that of a 
\(\kk_0\)-algebra map \(\kk\ra\kk[\eps]=\kk\Ten_{\kk_0}\kk_0[\eps]\) whose 
composition with the unique map \(\kk[\eps]\ra\kk\) is the identity.  
Geometrically, we are given an ``action'' \(\FY_0\x\spec(\kk)\ra\spec(\kk)\) 
of the scheme \(\FY_0=\spec(\kk_0[\eps])\), in such a way that the 
\(\kk_0\)-point of \(\FY_0\) acts as the identity. The finite scheme 
\(\FY_0\) is not a monoid, but it is part of a system defining an ind-scheme, 
the additive formal group \(\FY\), and in characteristic \(0\), each 
``action'' of \(\FY_0\) extends uniquely to an action of \(\FY\). As with the 
discrete case, the further components of the system \(\FY\) correspond to 
iterative application of the operator, i.e., to higher derivations (this 
example, which is classical, is discussed in detail in~\Sec{sec:ms}. We note 
that one could think of the additive formal group as a ``limit'' of the 
additive group of the integers, as the generator \(1\) ``tends to \(0\)'').

It is therefore reasonable to define a ``field with operators'' simply as a 
field with an arbitrary ind-finite scheme monoid action. This is 
(essentially) the approach taken in \citet{MR2602031,ms2}, and which we adopt 
here. We mentioned that similar ideas appear before: for example, 
\Cite[2.4]{buium} discusses the encoding of a certain class of operators by 
suitable algebra maps (In fact, the approach there is somewhat more general, 
see~\Sec{q:monads}). The case of \(\kk[\eps]\) goes back (at least) to Weil 
(unpublished), and is, in any case, classical. The case of Hasse--Schmidt 
derivations is discussed in~\Cite[27]{matsumura}. It appears that the 
geometric description we give is new (though~\Cite{ggo} appears to take a 
geometric approach of the differential case).

We name ind-finite schemes ``formal sets''. The guiding principle is that 
whatever can be done with usual sets should also be possible with formal 
sets. For example, for any set \(S\) there is a free monoid generated by 
\(S\). If \(S\) is a set of endomorphisms of \(\kk\), the resulting free 
monoid acts on \(\kk\) (and vice versa). The same is true for formal sets 
(Proposition~\ref{prp:free}). The free monoid generated by 
\(S=\spec(\kk_0[\eps])\) is the additive formal group precisely if \(\kk_0\) 
contains \(\QQ\). This explains why in characteristic \(0\) (and only in this 
case), a derivation is the same as an action of this formal group.

As another example, if \(S\) is a set and \(M\) is a (discrete) monoid, there 
is an ``induced'' \(M\)-set \(S^M\) of functions from \(M\) to \(S\), and 
this operation is right adjoint to the ``restriction'', from \(M\)-sets to 
sets, where one forgets the \(M\)-action. A similar construction is available 
when \(M\) is now a formal set, and \(S\) is a (nice) scheme. The resulting 
space is (essentially) what one calls the \emph{prolongation space} of \(S\), 
which we discuss  in~\Sec{ssc:prolong} (and which is defined originally 
in~\Cite{MS2}). One may then do ``\(M\)-geometry'' (analogous to Kolchin's 
differential algebraic geometry), where the prolongation spaces allow one to 
view usual algebraic schemes as \(M\)-schemes. In particular, it possible to 
consider linear groups, and their representations.

Recall that in the algebraic case, the category of representations of a 
(pro-) linear group scheme (over a field) is described by equipping the pure 
category structure with a tensor structure, satisfying suitable properties. A 
theorem of Saavedra (\Cite{Saavedra}) then shows that if one also remembers 
the ``forgetful'' functor into the category of vector spaces, the description 
is complete: any tensor category as above is equivalent to the category of 
representations of a linear group scheme, which can be recovered from the 
fibre functor (See~\DM{} for an exposition).

In the case of differential algebraic groups, the tensor structure is 
insufficient. For example, the multiplicative group \(\sG_m\), viewed as a 
differential algebraic group, admits a non-trivial differential algebraic 
homomorphism to the additive group \(\sG_a\), and \(\sG_a\) itself has the 
derivative as an endomorphism. To recover the differential algebraic group in 
this case, \Cite{ovchin} introduces a new operation on representations, as 
follows: One considers a \(2\)-dimensional \(\kk\)-vector space \(D\), which 
admits an additional vector space structure (on the right), coming from the 
derivation. Given a vector space (or a representation) \(V\), one obtains a 
new such object \(\tau(V)=D\hat{\Ten}V\), where the tensor product is with 
respect to the right structure (and the vector space structure on \(\tau(V)\) 
comes from the left one). It is then shown that with this additional 
structure, the differential algebraic group can be recovered.

In~\Cite[4]{tannakian}, the operation described above is abstracted to apply 
in an arbitrary tensor (abelian) category. This is done by defining the 
prolongation of a tensor category \(\cC\) as a certain tensor category of 
exact sequences of objects of \(\cC\). The differential structure is then 
given by a tensor functor \(\tau\) from \(\cC\) to its prolongation. While 
this formalism does work in the required way, the definition of the 
prolongation category is somewhat ad hoc. One of the goals of the current 
paper is to build the prolongation category in a more systematic way. The 
relation of the current formalism with the original one is explained in 
Example~\ref{ex:difftan}. We remark that an alternative formulation of the 
differential theory, including an extension to the case of several 
derivations, is also suggested in a recent preprint, \Cite{ggo}. Their 
approach seems to be similar to the one taken here, and it would be 
interesting to make a precise comparison.

In this paper, we imitate the differential case, but construct the 
prolongation category more systematically. Given a field \(\kk\) on which a 
formal monoid \(M\) (eventually) acts, we identify the analogue of the space 
\(D\) above: essentially, it is the dual of the pro-algebra corresponding to 
\(M\).  In~\Sec{sec:prolong}, we define and study the prolongation of an 
abstract tensor category over \(\kk\) with respect to \(M\). In this section, 
only \(M\) itself (or rather, its algebra) is used, the action does not yet 
appear. In~\Sec{sec:ms} we discuss the notion of a field with operators. This 
is mostly an exposition of parts of~\Cite{MS2}, although there are also new 
results (and the exposition is somewhat different).  Finally, 
in~\Sec{sec:main} we give our main result. We remark that, though the 
statement is completely analogous to the previous cases, the proof in this 
paper is different: rather than repeating a proof similar to the classical 
case, our current proof \emph{reduces} to the classical statement, using the 
description of \(M\)-schemes employed in this paper. As a result, the notion 
of schemes in \(\cC\), and the scheme structure on objects of \(\cC\) does 
not play a role as it played in the differential case (and less explicitly, 
in the classical case).  Nevertheless, we discuss this structure 
in~\Sec{ssc:cschemes}. In~\Sec{ssc:descent}, we prove an analogue of a result 
of Deligne (\Deligne{6.20}), which states that, under suitable assumptions, a 
fibre functor over a large field descends to a fibre functor over the base 
sub-field. This is used in the application to Galois theory.


In the process of writing the paper, I ran into some questions that are not 
essential for the main results, but they do occur naturally. I list some of 
these questions in~\Sec{sec:questions}.

\subsection*{Remarks on notation}
For most of the paper, \(\kk_0\) is a base ring, consisting of ``constants'' 
for the operators. Hence, all maps will be maps over \(\kk_0\). The (pro-) 
finite algebra corresponding to the acting monoid is usually denoted by 
\(E_0\), and the monoid itself by \(\FY_0\).  \(\kk\) will be an extension of 
\(\kk_0\) (usually a field), on which \(\FY_0\) acts.  I remove the subscript 
\(0\) when changing the base to \(\kk\): \(E=\kk\Ten_{\kk_0}E_0\), etc.

If \(E\) is a pro-algebra, \(\spec(E)\) is the ind-scheme given by applying 
\(\spec\) to a system representing \(E\) (thus, if \(E\) is a pro-finite 
algebra, \(\spec(E)\) is closely related to the formal spectrum of the 
pseudo-compact algebra, in the sense 
of~\Cite[Expose~VII\textsubscript{B}]{sga3}, obtained by taking the inverse 
limit on \(E\).  However, this point of view will not be very convenient for 
us).

By \(\kk[\eps]\) I mean the ring of dual numbers over \(\kk\) (so 
\(\eps^2=0\)). By \(\kk[[x]]\) I mean the pro-algebra given by the system 
\((\kk[x]/x^n)\), \(n>0\), with the reduction maps (and similarly for other 
power-series rings). Combining with the previous paragraph, 
\(\spec(\kk[[x]])\) is the ind-scheme represented by the system 
\((\spec(\kk[x]/x^n))\), which can also be thought of as the formal spec of 
the topological algebra \(\kk[[x]]\) (and I will never consider the spectrum 
of the algebra \(\kk[[x]]\)).

I try to use different font styles for different kind of objects. This should 
be visible.

\subsection{An application to Galois theory}\label{ssc:galois}
We briefly recall here how a formalism of the type discussed in the paper can 
be applied to defined Galois groups of linear equations. This is, 
essentially, a repetition of~\Deligne{\S~9} in our setting. It is provided 
here for the benefit of readers who like to have an application in mind, and 
will not be used later in the paper. We use the notation and conventions of 
the rest of the paper, especially those in~\Sec{sec:main}. For concreteness, 
we work with linear difference equations over an iterative Hasse--Schmidt 
field, but the same idea works in general. We note that we allow 
characteristic \(0\), in which case we recover the Galois theory of 
difference equations over a differential field, as in~\Cite{hardouinsinger}.

Let \(\kk\) be a field, \(\FY=\spec(\kk[[x]])\), and assume we are given 
commuting actions of \(\FY\) and an automorphism \(\sg\) on 
\(\sZ=\spec(\kk)\) (i.e., an action of \(\FY\x\ZZ\)). In other words, \(\kk\) 
is an iterative difference---Hasse--Schmidt-differential field. Since \(\sg\) 
commutes with the \(\FY\), the fixed field \(C_\kk\) is an iterative 
Hasse--Schmidt field (abbreviated HS-field from now on).

A linear difference equation \(\sg(x)=Ax\) over \(\kk\) corresponds (up to 
equivalence) to a difference module over \(\kk\), i.e., to a finite 
dimensional vector space \(M\) over \(\kk\), together with a \(\kk\)-linear 
map \(\sg:\kk\Ten_\sg{}M\ra{}M\): The set of solutions of the equation in an 
\(\FY\x\ZZ\)-extension \(A\) of \(\kk\) is identified (by choosing a suitable 
basis for \(M\)) with \(\Ehom_{\kk,\sg}(M,A)\). We are interested in defining a 
(HS-differential-algebraic) Galois group that measures the 
differential-algebraic relations among such solutions.

Let \(\cC\) be the category of difference modules over \(\kk\). This is a 
rigid tensor category over \(C_\kk\) (the dual \(\Co{M}\) is the 
\(\kk\)-vector-space dual, with \(\sg(\phi)(m)=\sg^{-1}(\phi(\sg(m)))\); we 
use here that \(\sg\) is invertible on \(\kk\)).

Given a difference module \(M\) and a number \(k\), we may define a new 
difference module \(\tau_kM\) as follows: \(\tau_kM\) is generated, as a 
\(\kk\)-vector space, by expressions \(\dd_im\), where \(0\le{}i\le{}k\), and 
\(m\in{}M\), subject to the relations
\begin{enumerate}
  \item \(\dd_i(m_1+m_2)=\dd_i(m_1)+\dd_i(m_2)\)
  \item \(\dd_i(am)=\sum_{j=0}^i\dd_j(a)\dd_{i-j}m\)
\end{enumerate}
for \(m_i\in{}M\) and \(a\in\kk\). The difference structure is given by 
\(\sg(\dd_im)=\dd_i\sg(m)\) (this is well defined, since \(\sg\) commutes 
with the \(\dd_i\) on \(\kk\)).

For each \(k\), \(\tau_kM\) has the structure of a module over 
\(E_k=\kk[x]/x^{k+1}\), given by \(x(\dd_im)=\dd_{i-1}m\) (with 
\(\dd_{-1}m=0\)). This modules structure commutes with \(\sg\), and is 
furthermore \(E_k\)-injective: it is enough to show that an element killed by 
\(x^l\) is of the form \(x^{k-l+1}m\), but this is obvious. With the evident 
definition on morphisms, we obtain a functor \(\tau_k\) from \(\cC\) to 
\(\Plng{E_k}\).  We give this functor a tensor structure by mapping 
\(\dd_i(m\Ten{}n)\in\tau_k(M\Ten{}N)\) to 
\(\sum_{j=0}^i\dd_jm\Ten\dd_{i-j}n\in\tau_k(M)\Ten^{E_k}\tau_k(N)\) (as in 
Example~\ref{ex:difftan}).

The functor \(i^!=i_k^!\), corresponding to the unique \(\kk\)-point 
\(E_k\ra\kk\), is given on modules of the form \(\tau_k(M)\) by sending 
\(\dd_jm\) to \(0\) for \(j>0\). The map \(m\mt\dd_0{}m\) is clearly an 
isomorphism \(a_M:M\ra{}i^!\circ\tau(M)\).

Similarly, \(\tau_k\tau_lM\) is generated (as a vector space) by elements 
\(\dd_i\dd_jm\), where \(i\le{}k\), \(j\le{}l\), and \(m\in{}M\).  
Furthermore, the elements \(\dd_k\dd_lm\) generate it as an \(E_k\Ten{}E_l\) 
module. We define an \(E_k\Ten{}E_l\)-module map from \(\tau_k\tau_lM\) to 
\(\Ehom_{E_{k+l}}(E_k\Ten{}E_l,\tau_{k+l}M)\) (a component of 
\(m^!\circ\tau_{k+l}\)) by sending the element \(\dd_k\dd_lm\) to the map 
\(x^sy^t\mt\binom{k+l-s-t}{k-s}\dd_{k+l-s-t}m\) (We have identified 
\(E_k\Ten{}E_l\) with \(\kk[x,y]/(x^{k+1},y^{l+1})\); note that the binomial 
coefficient is \(0\) when \(s>k\) or \(t>l\)). It is clear that both \(a\) 
and \(b\) commute with the action of \(\sg\). This completes the definition 
of \(\cC\) as an \(E\)-tensor category over \(C_\kk\) 
(Definition~\ref{def:etensor}).

The \(E\)-tensor category \(\cC\) admits a fibre functor \(\w\) over \(\kk\), 
namely, forgetting the action of \(\sg\). This fibre functor has an 
\(E\)-structure, given by identifying \(\dd_im\in\tau_kM\) with 
\(\theta_i\Ten{}m\) in \(\Co{E_k}\Ten_\mu{}M\), where \((\theta_i)\) is the 
basis of \(\Co{E_k}\) dual to \((x^i)\). Hence we have an \(E\)-fibre functor 
\(\w\) over \(\kk\).

Let now be \(\oX\) an object of \(\cC\), and let \(\cC_\oX\) be the 
\(E\)-tensor sub-category of \(\cC\) generated by \(\oX\). The fibre functor 
\(\w\) restricts to \(\cC_\oX\), so if \(C_\kk\) is HS-differentially closed 
(which, in the case of positive characteristic, is the same as 
separably-closed, of imperfection degree \(1\); cf.~\Cite{ziegler}), we are 
in the situation of Proposition~\ref{prp:descent}.  Hence, in this case, 
\(\cC_\oX\) admits a fibre functor over \(C_\kk\). The corresponding 
\(E\)-group, obtained from Theorem~\ref{thm:main}, is (by definition) the 
Galois group of \(\oX\) (or a corresponding difference equation).

\section{Categorical Prolongations}\label{sec:prolong}
In this section, our goal is to define the prolongation of a tensor category 
with respect to an algebra \(E\). This will be another tensor category, as 
described in the introduction. The goal is achieved in 
Definition~\ref{def:prolongf} in the finite case, and extended to the 
pro-finite case (which is the general case) in Definition~\ref{def:prolongp}.  
We start by discussing the action of \(E\) on objects in an arbitrary 
\(\kk\)-linear category. Much of the material comes from~\Deligne{5}.

\subsection{Modules in \(\kk\)-linear categories}
Let \(\kk\) be a field, fixed for the entire section. We denote by 
\(\Vec_\kk\) the category of finite dimensional vector spaces over \(\kk\).  
For any vector space \(V\) over \(\kk\), \(\Co{V}\) denotes the linear dual.  

If \(E\) is a finite \(\kk\)-algebra, then \(\Co{E}\) is an \(E\)-module. We 
note that \(\Co{E}\) need not be free:

\begin{example}
  Let \(E=\kk[x,y]\), with \(x^2=y^2=xy=0\).  Let 
  \((\delta,\delta_x,\delta_y)\in\Co{E}^3\) be the basis dual to the basis 
  \((1,x,y)\) of \(E\). Then \(x\delta=x\delta_y=0\) and similarly for \(y\), 
  so \(\Co{E}\) cannot be free.
\end{example}

The actual situation is described by the following.

\begin{prop}\label{prp:injflat}
  Let \(M\) and \(N\) be finite \(E\)-modules. Then \(\Co{(M\Ten_E{}N)}\) is 
  canonically isomorphic to \(\Ehom_E(M,\Co{N})\). In particular, \(N\) is flat 
  if and only if \(\Co{N}\) is injective (as with any commutative ring, this 
  is also equivalent to \(N\) being projective and locally free).
\end{prop}
\begin{proof}
  An element of \(\Ehom_E(M,\Co{N})\) corresponds, by adjunction, to a 
  \(\kk\)-linear map \(\phi:M\Ten{}N\ra\kk\), such that 
  \(\phi(em,n)=\phi(m,en)\) for all \(e\in{}E\). These are precisely the 
  elements of \(\Co{(M\Ten_E{}N)}\).
\end{proof}

We will be interested in modules in arbitrary \(\kk\)-linear categories.

\begin{defn}
  Let \(\kk\) be a field. By a \Def{\(\kk\)-linear category} we mean an 
  additive category \(\Cc\), together with a \(\kk\)-vector space structure 
  on each abelian group \(\Ehom(\oX,\oY)\), such that
  \begin{enumerate}
    \item Composition is \(\kk\)-bilinear.
    \item Each \(\Ehom(\oX,\oY)\) is finite dimensional
    \item Each object has finite length (i.e., the length of a strictly 
      descending chain starting from a given object is bounded).
  \end{enumerate}
\end{defn}

In particular, each \(\Eend(\oX)\) is a finite (associative) \(\kk\)-algebra.  

\point[Action of \(\Vec_\kk\) on \(\Cc\)]
For any object \(\oX\) of a \(\kk\)-linear category \(\Cc\), the functor 
\(\oY\mt{}\Ehom(\oX,\oY)\) (into \(\Vec_\kk\)) has a left adjoint 
\(V\mt{}V\Ten_\kk\oX\) (\DM{\S~2}). We note that for vector spaces \(V\) and 
\(W\), there is a canonical isomorphism (omitted from notation) 
\((V\Ten_\kk{}W)\Ten_\kk\oX=V\Ten_\kk(W\Ten_\kk\oX)\), since 
\(\Ehom(V\Ten_\kk{}W,\Ehom(\oX,\oY))\) is canonically identified with 
\(\Ehom(V,\Ehom(W,\Ehom(\oX,\oY)))\). Set 
\(\Hom_\kk(V,\oX)=\Co{V}\Ten_\kk\oX\).

\begin{prop}\label{pro:ladj}
  Let \(\Cc\) be a \(\kk\)-linear category, \(\oX\) an object of \(\Cc\).
  \begin{enumerate}
    \item The functor \(\oY\mt\Co{\Ehom(\oY,\oX)}\) (from \(\Cc\) to 
      \(\Vec_\kk\)) is left adjoint to \(V\mt{}V\Ten\oX\).
    \item For any vector space \(V\) finite dimensional over \(\kk\), 
      \(\Hom_\kk(V,-)\) is right adjoint to \(V\Ten_\kk-\)
    \item The functor \(-\Ten_\kk\oX\) is exact
    \item In \(\Op{\Cc}\), \(\Op{(V\Ten_\kk\oX)}\) is canonically isomorphic 
      to \(\Hom_\kk(V,\Op{\oX})\).
  \end{enumerate}
\end{prop}
\begin{proof}
  \begin{enumerate}
    \item Given A morphism \(f:\oY\ra{}V\Ten_\kk\oX\), we obtain a map
      \begin{equation}
        \Co{V}\Ten_\kk\oY\ra[id\Ten_\kk f]\Co{V}\Ten_\kk 
        V\Ten_\kk\oX\ra[ev\Ten_\kk id]\oX
      \end{equation}
      which induces, by adjunction, a map \(\Co{V}\ra{}\Ehom(\oY,\oX)\), and by 
      duality a map \(\Co{f}:\Co{\Ehom(\oY,\oX)}\ra{}V\).

      In the other direction, a map \(g:\Co{\Ehom(\oY,\oX)}\ra{}V\) corresponds 
      by duality to a map \(\Co{V}\ra{}\Ehom(\oY,\oX)\), which corresponds by 
      adjunction to a map \(\Co{V}\Ten_\kk\oY\ra\oX\). Tensoring with \(V\) 
      and combining with co-evaluation, we get
      \begin{equation}
        \Co{g}:\oY\ra V\Ten_\kk\Co{V}\Ten_\kk\oY\ra V\Ten_\kk\oX
      \end{equation}
      
      The statement that the two constructions are inverse to each other is 
      precisely the statement that the usual evaluation and co-evaluation 
      determine a rigid monoidal structure on \(\Vec_\kk\).
    \item Apply the previous statement with \(\Co{V}\)
      
    \item The functor has both left and right adjoints

    \item A morphism \(\oY\ra{}V\Ten_\kk\oX\) in \(\Op{\Cc}\) corresponds to 
      a morphism \(V\Ten_\kk\oX\ra\oY\) in \(\Cc\), which correspond to a 
      linear map \(V\ra{}\Ehom_{\Cc}(\oX,\oY)=\Ehom_{\Op{\Cc}}(\oY,\oX)\), 
      corresponding by duality to a map 
      \(\Co{\Ehom_{\Op{\Cc}}(\oY,\oX)}\ra\Co{V}\), which corresponds by 
      adjunction to a map \(\oY\ra\Hom_\kk(V,\oX)\) (in \(\Op{\Cc}\)). Hence 
      the two objects represent the same functor.
  \end{enumerate}
\end{proof}

\begin{defn}
  Let \(E\) be an associative \(\kk\)-algebra. A \Def{(left) \(E\)-module in 
  \(\Cc\)} is an object \(\oX\) of \(\Cc\), together with a \(\kk\)-algebra 
  map \(E\ra{}\Eend(\oX)\). We denote by \(\Lmod[\Cc]{E}\) the category of left 
  \(E\)-modules in \(\Cc\) (with \(E\)-action preserving maps).
\end{defn}

\point
We fix a finite-dimensional \(\kk\)-algebra \(E\). Then an \(E\)-module  
structure on \(\oX\) is the same as a morphism \(E\Ten\oX\ra\oX\) satisfying  
the obvious relations, and by Proposition~\ref{pro:ladj}, it is also the same 
as a morphism \(\oX\ra\Co{E}\Ten_\kk\oX=\Hom_\kk(E,\oX)\) that makes \(\oX\) 
an \(\Co{E}\)-comodule.

\point\label{pt:eten}
Given an \(E\)-module \(\oX\) in \(\Cc\) and an object \(\oY\) in \(\Cc\), 
the space \(\Ehom(\oX,\oY)\) is a (usual) right \(E\)-module. In other words, 
\(\oX\) represents a functor \(\oY\mt{}\Ehom(\oX,\oY)\) from \(\Cc\) to the 
(abelian, \(\kk\)-linear) category \(\Rcoh\) of finitely generated right 
\(E\)-modules.  This functor has a left adjoint \(\fF_\oX:M\mt{}M\Ten_E\oX\), 
where \(M\Ten_E\oX\) is the co-equaliser
\begin{equation}
  M\Ten E\Ten\oX\rightrightarrows M\Ten\oX\ra M\Ten_E\oX
\end{equation}

\begin{defn}\label{def:flat}
  An \(E\)-module \(\oX\) in \(\Cc\) is \Def{flat} if the functor 
  \(M\mt{}M\Ten_E\oX\) from \(\Rcoh\) to \(\Cc\) is exact. We denote by 
  \(\Flat{E}\) the full sub-category of \(\Lmod{E}\) consisting of flat 
  \(E\)-modules.
\end{defn}

\begin{lemma}
  If for any right ideal \(I\) of \(E\), the map \(I\Ten_E\oX\ra\oX\) is 
  monic, then \(\oX\) is flat.
\end{lemma}
\begin{proof}
  Same as for usual modules
\end{proof}

\begin{example}\label{ex:diff}
  Let \(\DD(\Cc)\) be the category of exact sequences 
  \(0\ra\oX\ra[i]\oY\ra[\pi]\oX\ra{}0\) in \(\Cc\), where the morphisms are 
  morphisms of exact sequences, in which the two side maps agree. This 
  category can be identified with \(\Flat{E}\), for \(E=\kk[\eps]\) (with 
  \(\eps^2=0\)). Namely, a sequence as above is identified with \(\oY\), with 
  \(\eps\) acting as \(i\circ\pi\) (and the sequence is exact precisely if 
  \(\oY\) is flat).
\end{example}

\begin{example}\label{ex:prod1}
  If \(E=E_1\x{}E_2\), then \(\Lmod{E}\) can be identified with 
  \(\Lmod{E_1}\x\Lmod{E_2}\), and likewise for \(\Flat{E}\). In particular, 
  for \(E=\kk\x\kk\), both \(\Lmod{E}\) and \(\Flat{E}\) are the category of 
  pairs of objects of \(\Cc\).
\end{example}

\begin{prop}[\Deligne{5.2}]
  Given an \(E\)-module \(\oX\) in \(\Cc\), the functor \(\fF_\oX\) 
  from~\ref{pt:eten} is right-exact. The assignment \(\oX\mt\fF_\oX\) is an 
  equivalence between \(\Lmod{E}\) and right-exact, \(\kk\)-linear functors 
  from \(\Rcoh\) to \(\Cc\). Flat modules correspond to exact functors under 
  this equivalence.
\end{prop}
For convenience, we sketch the proof.
\begin{proof}
  Given a right-exact functor \(\fF:\Rcoh\ra\Cc\), let \(\oX=\fF(E)\).  Since 
  \(E=\Eend_E(E)\) (endomorphisms of right \(E\)-modules), \(\oX\) is a left 
  \(E\)-module.  Given any coherent right \(E\)-module \(M\), applying 
  \(\fF\) to the co-equaliser diagram
  \begin{equation}
    M\Ten_\kk E\Ten_\kk E\rightrightarrows M\Ten_\kk E\ra M\Ten_E E=M
  \end{equation}
  and using the fact that \(\fF\) is right-exact and \(\kk\)-linear, we get a 
  co-equaliser diagram
  \begin{equation}
    M\Ten_\kk E\Ten_\kk\oX\rightrightarrows M\Ten_\kk\oX\ra M\Ten_E\oX=\fF(M)
  \end{equation}
  Hence \(\fF\) is isomorphic to \(\fF_\oX\). The other claims are obvious.
\end{proof}

\point
We will require a few more results from~\Deligne{5}.  Given a left 
\(E\)-module \(M\) and an object \(\oX\) of \(\Cc\), \(M\Ten\oX\) is 
naturally an \(E\)-module in \(\Cc\). Hence, there is a functor 
\(\Ten:\Lcoh\x\Cc\ra\Lmod{E}\), \(\kk\)-linear and exact in each coordinate 
(where \(\Lcoh\) is the category of finite left \(E\)-modules).  

\Deligne{5.1} defines the tensor product of two abelian \(\kk\)-linear 
categories \(\Cc_1\) and \(\Cc_2\) to be an abelian \(\kk\)-linear category 
\(\Cc=\Cc_1\Ten_\kk\Cc_2\), together with a ``universal'' \(\kk\)-bilinear 
right-exact (in each coordinate) functor \(\Ten:\Cc_1\x\Cc_2\ra\Cc\).

\begin{prop}[\Deligne{5.11}]\label{prp:deligne1}
  The functor \(\Ten:\Lcoh\x\Cc\ra\Lmod{E}\) identifies \(\Lmod{E}\) with the 
  tensor product of \(\Lcoh\) and \(\Cc\).
\end{prop}

\begin{prop}[\Deligne{5.13}]\label{prp:deligne2}
  Let \(\Cc_1\) and \(\Cc_2\) be two abelian \(\kk\)-linear categories.
  \begin{enumerate}
    \item \(\Cc=\Cc_1\Ten_\kk\Cc_2\) exists, and is again \(\kk\)-linear
    \item The ``tensor product'' \(\Ten:\Cc_1\x\Cc_2\ra\Cc\) is exact in each 
      coordinate.
    \item If \(\kk\) is perfect, and \(\fF:\Cc_1\x\Cc_2\ra\Dd\) is exact in 
      each coordinate, then the induced functor \(\Cc\ra\Dd\) is exact as 
      well. \qedhere
  \end{enumerate}
\end{prop}

\begin{remark}\label{rmk:nonperf}
  The assumption that \(\kk\) is perfect in the last part above of 
  Proposition~\ref{prp:deligne2} is too strong for our purposes, since it 
  excludes some interesting examples. The fact that the statement is false in 
  general is demonstrated in~\Deligne{\S~5.6}. However, the statement remains 
  true in the following situation: We call a \(\kk\)-algebra \(E\) 
  \Def{quasi-separable} over \(\kk\) if it is commutative, finite over 
  \(\kk\), and the associated reduced algebra is separable over \(\kk\). If 
  \(E\) is such an algebra, then the last part of 
  Proposition~\ref{prp:deligne2} holds when \(\Cc_1=\Lcoh\) (and \(\Cc_2\) is 
  any), without restriction on \(\kk\).

  The proof of this fact is the same as for the original statement. One 
  reduces, as in~\Deligne{\S~5}, to the following statement 
  (\Deligne{\S~5.9}): If \(E\) is quasi-separable, and \(A\) is any finite 
  (not necessarily commutative) algebra over \(\kk\), \(S\) and \(T\) are 
  simple modules over \(E\) and \(A\), respectively, then \(S\Ten_\kk{}T\) is 
  a semi-simple \(E\Ten_\kk{}A\)-module. To prove this, we note that since 
  \(S\) is simple, any nilpotent element of \(E\) acts as \(0\), so we may 
  assume that \(E\) is reduced, and hence separable. Now the proof proceeds 
  as in~\Deligne{5.9}.
\end{remark}

\subsection{Tensor structure}\label{ssc:tensor}
We now assume that \(\kk\) is a field, \(E\) is a quasi-separable 
\(\kk\)-algebra (Remark~\ref{rmk:nonperf}), and \(\Cc\) is abelian and 
\(\kk\)-linear.  We also assume  that we are given a monoidal structure 
\((\Ten,\phi,\psi)\) on \(\Cc\) (so that \(\Ten\) is \(\kk\)-linear in each 
coordinate, and has a unit \(\1\), \(\phi\) and \(\psi\) are, respectively, 
associativity and commutativity constraints, but \(\Cc\) is not necessarily 
rigid).  We assume \(\Ten\) to be exact in each coordinate (this is automatic 
if \(\Cc\) is rigid).  We would like to define a monoidal structure on 
\(\Lmod{E}\).  It will be convenient to define and work with two dual such 
structures.

We fix a unit \(\1\) in \((\Cc,\Ten)\). The functor \(V\mt{}V\Ten_\kk\1\) has 
a natural tensor structure, making it a fully faithful exact tensor embedding 
of \(\Vec_\kk\) into \(\Cc\). We will therefore view \(\Vec_\kk\) as a 
subcategory of \(\Cc\). The meaning of all notions we have defined (and will 
define) for both vector spaces and objects of \(\Cc\) is easily seen to be 
compatible with this identification. For example, we have 
\(V\Ten_\kk\oX=V\Ten\oX\) and \(\Hom_\kk(V,\oX)=\Hom(V,\oX)\) (in particular, 
the latter exists), so we drop the decoration \(\kk\) from now on.

\point\label{pt:base}
Given two \(E\)-modules \(\oX\) and \(\oY\) in \(\Cc\), their usual tensor 
product \(\oX\Ten_E\oY\) is defined as the largest quotient of \(\oX\Ten\oY\) 
on which the two actions of \(E\) agree (cf.~\DM{3}). In other words, it is 
the co-equaliser
\begin{equation}\label{eq:tene}
  E\Ten\oX\Ten\oY\rightrightarrows\oX\Ten\oY\ra\oX\Ten_E\oY
\end{equation}
The \(E\)-module structure is induced, as usual, by the action on either 
coordinate.

The dual tensor product \(\oX\Ten^E\oY\) is defined as 
\(\Op{(\Op{\oX}\Ten_E\Op{\oY})}\), where \(\Op{\oX}\) is \(\oX\) viewed as an 
object of the opposite category \(\Op{\Cc}\) (since \(E\) is commutative, 
\(\Lmod[\Op{\Cc}]{E}=\Op{(\Lmod{E})}\)). In other words, it is the largest 
sub-object of \(\oX\Ten\oY\) annihilated by all maps \(e\Ten{}1-1\Ten{}e\) 
with \(e\in{}E\) (this exists since \(E\) is finite). Again, the \(E\)-module 
structure comes from the action on either coordinate.

Since the associativity and commutativity constraints are functorial, they 
commute with the action of \(E\), and therefore induce similar constraints  
\(\phi_E\), \(\psi_E\), \(\phi^E\) and \(\psi^E\) on the respective tensor 
structures. We set \(\Cc_E=(\Lmod{E},\Ten_E,\phi_E,\psi_E)\) and 
\(\Cc^E=(\Lmod{E},\Ten^E,\phi^E,\psi^E)\).

\begin{lemma}\label{lem:dual}
  The inclusion of \(\oX\Ten^E\oY\) in \(\oX\Ten\oY\) is the equaliser of the 
  two maps \(\oX\Ten\oY\ra\Hom(E,\oX\Ten\oY)\).
\end{lemma}
\begin{proof}
  This follows from dualising the diagram~\eqref{eq:tene}, using 
  Proposition~\ref{pro:ladj}.
\end{proof}

The following proposition lists the basic properties of these operations.
\begin{prop}\label{prp:modten}
  Let \((\Cc,\Ten)\) and \(E\) be as in~\rpt{pt:base}.
  \begin{enumerate}
    \item \(\Cc_E\) and \(\Cc^E\) are monoidal categories
    \item If \(\Cc\) is closed, then so is \(\Cc_E\).
    \item If \(\Cc\) is rigid, then
      \begin{equation}\label{eq:dual}
        \Co{(\Co{\oX}\Ten_E\Co{\oY})}=
        \oX\Ten^E\oY=
        \Hom_E(\Co{\oX},\oY)
      \end{equation}
  \end{enumerate}
  Hence, if \(\Cc\) is rigid, \(\oX\mt\Co{\oX}\) induced a monoidal 
  equivalence \(\Cc^E\ra\Op{\Cc_E}\).
\end{prop}
\begin{proof}
  \begin{enumerate}
    \item This was discussed in~\rpt{pt:base}. The only additional point is 
      that \(E\) is a unit for \(\Cc_E\), and dually, \(\Co{E}\) is a unit in 
      \(\Cc^E\).
    \item Given two \(E\)-modules \(\oX\) and \(\oY\) in \(\Cc\), \(E\) acts 
      on \(\Hom(\oX,\oY)\) in two ways. Let \(\Hom_E(\oX,\oY)\) be the 
      equaliser of the two actions, with \(E\) structure coming from either.  

      A map \(f:\oZ\ra\Hom_E(\oX,\oY)\) determines a map 
      \(\oZ\ra\Hom(\oX,\oY)\), and therefore a map \(g:\oZ\Ten\oX\ra\oY\). If 
      \(\oZ\) is an \(E\)-module, and \(f\) commutes with the action of 
      \(E\), then the two compositions 
      \(E\Ten\oZ\Ten\oX\ra\oZ\Ten\oX\ra[g]\oY\) are equal, so \(g\) descends 
      to a map \(\bar{g}:\oZ\Ten_E\oX\ra\oY\). Furthermore, since \(f\) 
      factors through \(\Hom_E(\oX,\oY)\), \(\bar{g}\) is a map of 
      \(E\)-modules. The argument in the other direction is similar.

    \item
      The first equality follows from the fact that \(\oX\mt\Co{\oX}\) is an  
      exact tensor equivalence of \(\Cc\) with \(\Op{\Cc}\), taking 
      \(E\Ten\oX\) to \(\Hom(E,\Co{\oX})\) (and using Lemma~\ref{lem:dual}).

      The second equality follows from Lemma~\ref{lem:dual} and the 
      construction of \(\Hom_E(\Co{\oX},\oY)\) as an equaliser (together with 
      the isomorphism \(\Co{\oX}\Ten\oY=\Hom(\oX,\oY)\) in any rigid 
      category).
  \end{enumerate}
\end{proof}

\begin{remark}
  The equivalence mentioned in the Proposition does not imply that \(\Cc^E\) 
  is closed (which is generally false), since the notion of a closed category 
  is not self-dual (However, see~\ref{cor:rigid}).
\end{remark}

\point[Flatness]
If \(\oX\) is an \(E\)-module in \(\Cc\), the functor \(\oY\mt\oY\Ten_E\oX\) 
is always right exact (since it is a co-equaliser). We would like to consider 
those modules for which the functor is exact. Since the usual \(E\)-modules 
are included in \(\Lmod{E}\), each such module is flat in the sense 
of~\ref{def:flat}.  It follows from Deligne's result that flatness is 
sufficient for the exactness of this functor in general.

\begin{prop}\label{prp:flatness}
  For any flat \(E\)-module \(\oX\), the functor \(-\Ten_E\oX\) (from 
  \(\Cc_E\) to itself) is exact.
  
  If \(\Cc\) is rigid, then this is also equivalent to the exactness of 
  \(\oY\mt\Hom_E(\oY,\Co{\oX})\).
\end{prop}

The last part is an analogue of Proposition~\ref{prp:injflat}.

\begin{proof}
  It is enough to prove that \(-\Ten_E\oX\) is exact as a functor from 
  \(\Lmod{E}\) to \(\Cc\). According to Proposition~\ref{prp:deligne1}, 
  \(\Lmod{E}\) can be identified with \(\Lcoh\Ten_\kk\Cc\). Since \(E\) is 
  quasi-separable, it is enough, by Remark~\ref{rmk:nonperf}, to prove that 
  the induced functor \(\Lcoh\x\Cc\ra\Cc\) is exact in each coordinate.  This 
  induced functor is given by \((M,\oY)\mt(M\Ten_E\oX)\Ten\oY\), so precisely 
  equivalent to the flatness of \(\oX\) (recall that \(\Ten\) was assumed to 
  be exact).

  The second statement follows from the first together with 
  equation~\eqref{eq:dual}.
\end{proof}

From now on, \emph{we assume that \(\Cc\) is rigid}.

\begin{cor}\label{cor:flatness}
  The full sub-category \(\Flat{E}\) of \(\Cc_E\) consisting of flat modules 
  is a tensor sub-category (which need not be abelian). So is the full 
  sub-category of \(\Cc^E\) consisting of objects \(\oX\) for which 
  \(\Co{\oX}\) is flat.
\end{cor}
\begin{proof}
  We need only to prove that if \(\oX\) and \(\oY\) are flat, then so is 
  \(\oX\Ten_E\oY\). Hence we need to prove that the functor 
  \(M\mt{}M\Ten_E(\oX\Ten_E\oY)\) is exact. Since \(\Cc_E\) is a tensor 
  category, the last object is equal to \((M\Ten_E\oX)\Ten_E\oY\), so this is 
  a composition of two exact functors (using Proposition~\ref{prp:flatness} 
  for \(\oY\)).
\end{proof}

\begin{defn}\label{def:prolongf}
  Let \(\kk\) be a field, \(\Cc\) a rigid abelian \(\kk\)-linear tensor 
  category, and \(E\) a quasi-separable \(\kk\)-algebra 
  (Remark~\ref{rmk:nonperf}). An object \(\oX\) of \(\Lmod{E}\) will be 
  called \Def{\(E\)-injective} if \(\Co{\oX}\) is \(E\)-flat.
  
  The \Def{\(E\)-prolongation} of \(\Cc\), \(\Plng[\Cc]{E}\), is defined to 
  be the full tensor sub-category of \(\Cc^E\) consisting of \(E\)-injective 
  modules.
\end{defn}

\begin{remark}
  If \(\oX\) is \(E\)-injective, it follows that the functor 
  \(M\mt\Hom_E(M,\oX)\) (from \(\Lcoh\) to \(\Cc\)) is exact. The converse is 
  also true, using the same argument as in the proof of 
  Proposition~\ref{prp:flatness}. Hence, the notion of \(E\)-injective 
  objects can also be defined on the level of abelian categories, without 
  mentioning the tensor structure. Also, as with flatness, it is enough to 
  check the exactness on inclusions of an ideal of \(E\) in \(E\). On the 
  other hand, being \(E\)-injective is not the same as being an injective 
  object in \(\Lmod{E}\).

  Similarly, it follows from equation~\eqref{eq:dual} that an object \(\oX\) 
  is \(E\)-flat if and only if it is \(E\)-projective, in the sense that 
  \(\Hom_E(\oX,-)\) is exact, either on \(\Cc_E\) or on \(\Lcoh\) (and this 
  is again different from being projective in \(\Lmod{E}\)).
\end{remark}

\begin{cor}\label{cor:flatop}
  The tensor equivalence \(\oX\mt\Co{\oX}\) from Proposition~\ref{prp:modten} 
  induces a tensor equivalence \(\Plng{E}\ra\Op{\Flat{E}}\).
\end{cor}

\begin{example}\label{ex:difftan}
  Let \(E=\kk[\eps]\). In Example~\ref{ex:diff}, we have already identified 
  the flat \(E\)-modules in \(\Cc\) with exact sequences 
  \(0\ra\oX\ra[i]\oY\ra[\pi]\oX\ra{}0\) in \(\Cc\). Since the dual module 
  corresponds to the dual exact sequence, an \(E\)-module is \(E\)-flat if 
  and only if it is injective. Thus, this is also the category of injective 
  \(E\)-modules.

  Let \(0\ra\oX_1\ra[i]\oY_1\ra[\pi]\oX_1\ra{}0\) and 
  \(0\ra\oX_2\ra[i]\oY_2\ra[\pi]\oX_2\ra{}0\) be two exact sequences. The 
  inclusions induce inclusions 
  \(0\ra\oX_1\Ten\oY_2\ra[i\Ten{}1]\oY_1\Ten\oY_2\) and 
  \(0\ra\oY_1\Ten\oX_2\ra[1\Ten{}i]\oY_1\Ten\oY_2\), and therefore a map 
  \begin{equation*}
    \oX_1\Ten\oY_2\oplus\oY_1\Ten\oX_2\ra[i\Ten 1-1\Ten i]\oY_1\Ten\oY_2
  \end{equation*}
  whose kernel (by a simple diagram chase) is \(\oX_1\Ten\oX_2\). Taking the 
  quotient by this kernel, we therefore obtain a sub-object \(\oZ\) of 
  \(\oY_1\Ten\oY_2\). The equaliser \(\oW\) of the two maps \(\pi\Ten{}1\) 
  and \(1\Ten{}\pi\) is clearly a sub-object of \(\oZ\), there is an exact 
  sequence \(0\ra\oX_1\Ten\oX_2\ra\oW\ra\oX_1\Ten\oX_2\ra{}0\), which was 
  defined in~\Cite{tannakian} to be the tensor product of the two given 
  sequences.

  Viewing the \(\oY_i\) as \(E\)-injective modules, with \(\eps=i\circ\pi\) 
  on each \(\oY_i\), the equaliser of \(\eps\Ten{}1\) and \(1\Ten{}\eps\) 
  coincides with the equaliser of \(\pi\Ten{}1\) and \(1\Ten\pi\), so the 
  above definition coincides with (the exact sequence corresponding to) the 
  \(E\)-module \(\oY_1\Ten^E\oY_2\). Hence \(\Plng{E}\) coincides, as a 
  tensor category, with what was called the prolongation of \(\Cc\) 
  in~\Cite{tannakian}.
\end{example}

\begin{example}
  When \(E=E_1\x{}E_2\), and we identify \(\Lmod{E}\) with 
  \((\Lmod{E_1})\x(\Lmod{E_2})\), as in Example~\ref{ex:prod1}, all notions 
  again work component wise. Hence, for \(E=\kk\x\kk\), \(\Plng{E}\) is 
  \(\Cc\x\Cc\), as a tensor category.
\end{example}

\begin{prop}\label{prp:rigid}
  Let \(\Cc\) be rigid, and let \(\oX\) be a flat \(E\)-module in \(\Cc\). We 
  set \(\oX^*=\Hom_E(\oX,E)\).
  \begin{enumerate}
    \item
      \(\oX^*\) is \(E\)-flat.
    \item
      For any \(E\)-module \(\oY\) in \(\Cc\), the canonical map 
      \(\oX^*\Ten_E\oY\ra\Hom_E(\oX,\oY)\) is an isomorphism.
  \end{enumerate}
\end{prop}
\begin{proof}
  Both parts follow from the following special case of the second part.
  \begin{claim}
    For any (usual) coherent \(E\)-module \(M\), the canonical map 
    \(\oX^*\Ten_E{}M\ra\Hom_E(\oX,M)\) is an isomorphism.
  \end{claim}
  \begin{proof}[Proof of claim]
    \(M\) has a finite free resolution,
    \begin{equation*}
      0\ra E^{n_k}\ra\dots\ra E^{n_2}\ra E^{n_1}\ra M\ra 0
    \end{equation*}
    Applying the (exact) functor \(\Hom_E(\oX,-)\) to the sequence, we get an 
    exact sequence
    \begin{equation*}
      0\ra {(\oX^*)}^{n_k}\ra\dots\ra {(\oX^*)}^{n_2}\ra {(\oX^*)}^{n_1}\ra 
      \Hom_E(\oX,M)\ra 0
    \end{equation*}
    On the other hand, \(\oX^*\Ten_E{}M\) is, by definition, the co-kernel of 
    the map \({(\oX^*)}^{n_2}\ra {(\oX^*)}^{n_1}\) in that sequence.
  \end{proof}
  
  We now return to the proof of the Proposition.
  \begin{enumerate}
    \item
      We need to prove that the functor \(\oX^*\Ten_E{}-:\Lcoh\ra\Cc\) is 
      exact. According to the claim, this functor coincides with 
      \(\Hom_E(\oX,-)\), and since \(\oX\) is \(E\)-flat, the result follows.
    \item
      The two exact functors \(\Hom_E(\oX,-)\) and \(\oX^*\Ten_E{}-\) from 
      \(\Lmod{E}\ra\Cc\) restrict, according to 
      Proposition~\ref{prp:deligne1}, to functors on \(\Lcoh\x\Cc\), and it 
      is enough to show that they coincide on this category. The former 
      restricts to the functor \((M,\oY)\mt\Hom_E(\oX,M)\Ten\oY\), while the 
      latter to \((M,\oY)\mt(\oX^*\Ten_E{}M)\Ten\oY\). Hence the functors are 
      isomorphic by the claim.\qedhere
  \end{enumerate}
\end{proof}

\begin{cor}\label{cor:rigid}
  Assume that \(\Cc\) is rigid. Then so are \(\Flat{E}\) and \(\Plng{E}\).
\end{cor}
\begin{proof}
  From Proposition~\ref{prp:rigid} together with 
  Corollary~\ref{cor:flatness}, we conclude that \(\Hom_E(\oX,\oY)\) is flat 
  whenever \(\oX\) and \(\oY\) are. Since it clearly satisfies the adjunction  
  property, the rigidity of \(\Flat{E}\) follows from the second part 
  Proposition~\ref{prp:rigid} and~\Deligne{2.3,2.5}.

  For \(\Plng{E}\), the statement follows from Corollary~\ref{cor:flatop}, 
  since the opposite of a rigid category is rigid. We mention only that 
  \begin{equation}
    \Hom^E(\oX,-):=\Co{\oX}\Ten_E-
  \end{equation}
  is the right adjoint to \(\oX\Ten^E-\) in \(\Plng{E}\).
\end{proof}

\subsection{Passing to the limit}\label{ssc:fibred}
We would like now to replace the finite algebra \(E\) by a pro-finite one.  
This is done, essentially, by glueing a matching sequence of flat or 
injective modules along a filtering system.

\point[Limits of categories]
Let \(\pi:\cC\ra\cI\) be a fixed functor. We say that an object \(\oX\) of 
\(\cC\) is over an object \(\oJ\) of \(\cI\) if \(\pi(\oX)=\oJ\) (and 
likewise for morphisms). The \emph{fibre} \(\cC_\oJ\) of \(\cC\) (or \(\pi\)) 
over \(\oJ\) is the sub-category of \(\cC\) consisting of objects over 
\(\oJ\) and morphisms over the identity of \(\oJ\).

Recall (say, from~\DM{Appendix}) that \(\cC\) (or \(\pi\))  is a \Def{fibred 
category}  if for any map \(f:\oI\ra\oJ\) in \(\cI\), and any object \(\oX\) 
over \(\oJ\), there is a universal map over \(f\) in \(\cC\) from an object 
\(f^*(\oX)\) over \(\oI\), and furthermore, \((gf)^*(\sX)=f^*(g^*(\sX))\) for 
all \(f:\oI\ra\oJ\), \(g:\oJ\ra\oK\) (more precisely, we are given functorial 
isomorphisms between the two, satisfying pentagon identities).

In particular, \(f^*\) is a functor from \(\cC_\oJ\) to \(\cC_\oI\).  
Conversely, given a collection \(\cC_\oJ\) of categories, one for each object 
\(\oJ\) of \(\cI\), and functors \(f^*\) for morphisms \(f\) of \(\cI\), with 
compatible isomorphisms as above, one constructs a fibred category with the 
prescribed fibres and pullbacks. In this sense, a fibred category can be 
thought of as a presheaf of categories.

Given two fibred categories \(\cC\) and \(\cD\) over \(\cI\), a 
\Def{Cartesian functor} from \(\cC\) to \(\cD\) is a functor 
\(\fF:\cC\ra\cD\) over \(\cI\), together with functorial identifications 
\(g^*\fF(\sX)=\fF(g^*\sX)\) for all morphisms \(g\) in \(\cI\). A morphism 
between Cartesian functors is a morphism of functors over the identity on 
\(\cI\), which commutes with the identifications.

We view \(\cI\) as a fibred category over itself, via the identity functor.  
More generally, for any category \(\cD\), we have a fibred category 
\(\cD\x\cI\) over \(\cI\).

\begin{defn}
  Let \(\pi:\cC\ra\cI\) be a fibred category. The \Def{inverse limit} 
  \(\Lim[\cI]{\cC}\) is the category of Cartesian functors from \(\cI\) to 
  \(\cC\).
\end{defn}

Hence, an object of \(\Lim[\cI]{\cC}\) is given by a collection of objects 
\(\sX_\oJ\) of \(\cC_\oJ\), one for each object \(\oJ\) of \(\cI\), together 
with, for each morphism \(f:\oI\ra\oJ\) in \(\cI\), an isomorphism 
\(\sX_\oI\ra{}f^*(\sX_\oJ)\), such that the system of such isomorphisms is 
compatible with compositions. In particular, if \(\cI\) has a terminal object 
\(\1\), then the assignment \((\sX_\oJ)\mt\sX_\1\) is an equivalence of 
categories \(\Lim[\cI]{\cC}\iso\cC_\1\).

Intuitively, one may think of objects of \(\cI\) as pieces of some geometric 
objects, and of the morphisms as glueing instructions. The category \(\cC\) 
can be viewed as objects of a particular kind (say, vector bundles) over 
these pieces. An object of \(\Lim[\cI]{\cC}\) can then be viewed as an object 
of the same kind on the (hypothetical) glued space.

We note that \(\Lim[\cI]{\cC}\) satisfies the expected universal property: 
The category of functors from \(\cD\) to \(\Lim[\cI]{\cC}\) is equivalent to 
the category whose objects are compatible collections of functors 
\(\cD\ra\cC_\oJ\) (in other words, to the category of Cartesian functors 
\(\cD\x\cI\ra\cC\)).

\point
As before, when the \(\Ehom\) sets are abelian groups or \(\kk\)-vector spaces, 
we assume that the pullback functors preserve this structure.  We note that 
the limit of abelian categories need not be abelian in general (see 
also~\ref{q:ablim}).  Also, the limit of \(\kk\)-linear categories need not 
be \(\kk\)-linear in our definition, since the finiteness conditions need not 
hold. However, when \(\cI\) is filtering (which is the case of interest for 
us), we may think of \(\Ehom(\sX,\sY)\) as a pro-finite \(\kk\)-vector space, 
and composition is a morphism in this category.  Furthermore, each object has 
pro-finite length.  We may call such categories \Def{pro-\(\kk\)-linear}.

\point[Tensor structure]\label{pt:limten}
Assume now that each \(\cC_\oJ\) is a monoidal category, that the functors 
\(f^*\) are given with monoidal structure, and that the monoidal  structures 
are compatible, and compatible with the composition isomorphisms (This is 
equivalent to saying that we are given a \emph{Cartesian} functor 
\(\Ten:\cC\x_\cI\cC\ra\cC\), etc.). Then the limit category 
\(\Lim[\cI]{\cC}\) also has a monoidal structure, given pointwise. If each 
\(\cC_\oJ\) admits internal \(\Ehom\)s, and each \(f^*\) is closed, then the 
limit category again admits internal \(\Ehom\)s. Finally, if each \(\cC_\oJ\) 
is rigid, then so is the limit category (note that in this case, pullbacks 
are automatically closed,~\DM{Prop.~1.9}).

\point[The fibre-wise opposite]\label{pt:limop}
Given a fibred-category \(\pi:\cC\ra\cI\), each pullback functor \(f^*\) 
determines a functor between the opposite categories. It is easy to see that 
the data of the fibred-category determines a fibred-category data on the 
collection of opposite categories, which we call the opposite fibred category 
\(\Op{\cC}\). Hence we are given a functor \(\pi:\Op{\cC}\ra\cI\), and the 
underlying category \(\Op{\cC}\) is \emph{not} the opposite category in the 
usual sense, but we shall never consider the latter in this setting, so there 
is no room for confusion.

When the fibres are abelian, or monoidal, or rigid, then so are the 
opposites, and if the original data came (e.g.) from a fibred monoidal 
category, then the opposite is again fibred monoidal. In particular, if each 
fibre is rigid, then the assignment \(\oX\mt\Co{\oX}\) is a Cartesian tensor 
equivalence between \(\cC\) and \(\Op{\cC}\).

We also note that the limit of the opposite category is the opposite of the 
limit category. This is true even including the monoidal structure.

\point[Pullbacks for modules]\label{pt:pullbacks}
We now assume that we are given a category \(\cC\) as in~\ref{ssc:tensor}. It 
will be convenient to think about \(E\)-modules in \(\cC\) geometrically, 
just like with usual modules. Thus, an \(E\)-module in \(\cC\) is thought of 
as a family of objects of \(\cC\), parametrised by \(\spec(E)\), a flat 
module corresponds to a bundle of such objects, etc.

Given a map from a finite \(\kk\)-algebra \(E\) to another such algebra 
\(F\), corresponding to a map \(f:\spec(F)\ra\spec(E)\) (over \(\kk\)), we 
have functors \(f^*,f^!:\Lmod{E}\ra\Lmod{F}\), and a functor 
\(f_*:\Lmod{F}\ra\Lmod{E}\), given by \(f^*(\oX)=F\Ten_E\oX\), 
\(f^!(\oX)=\Hom_E(F,\oX)\) and \(f_*(\oX)\) is \(\oX\) viewed as an 
\(E\)-module via \(f\).  It follows directly from the definitions that 
\(f^*\) is left adjoint to \(f_*\), which is left adjoint to \(f^!\). Also, 
given another map \(g:\spec(G)\ra\spec(F)\), there are obvious isomorphisms 
\(g^*\circ{}f^*\ra(fg)^*\) and \(g^!\circ{}f^!\ra(fg)^!\).  Therefore, the 
determine two fibred categories \(\cC^*\) and \(\cC^!\) over the category 
\(\cS\) of finite schemes over \(\kk\), with pullbacks given by \(f^*\) and 
\(f^!\), respectively.

\begin{prop}\label{prp:pullbacks}
  Assume \(\cC\) is rigid, and let \(f:\spec(F)\ra\spec(E)\) be a map over 
  \(\kk\).
  \begin{enumerate}
    \item For any object \(\oX\) of \(\Lmod{E}\), there are canonical 
      isomorphisms \(f^!(\Co{\oX})=\Co{(f^*(\oX))}\)
    \item If \(\oX\) is \(E\)-flat, then \(f^*(\oX)\) is \(F\)-flat.
    \item If \(\oX\) is \(E\)-injective, then \(f^!(\oX)\) is 
      \(F\)-injective.
  \end{enumerate}
\end{prop}
\begin{proof}
  \begin{enumerate}
    \item The isomorphism is given by Equation~\eqref{eq:dual} of 
      Proposition~\ref{prp:modten}. Since all constructions are functorial, 
      it commutes with the \(F\)-action.
    \item The same as for usual modules
    \item By the first two parts\qedhere
  \end{enumerate}
\end{proof}

We note that the tensor structure was not used in any essential way (the 
duality could be replaced by passing to the opposite category). On the other 
hand, given the tensor structure, the restriction functors are tensor functor 
for the corresponding structure, in the obvious way, and the canonical 
isomorphisms are tensor isomorphisms.

It follows from the proposition that the fibred categories \(\cC^*\) and 
\(\cC^!\) above contain fibred sub-categories \(\cC^f\) and \(\cC^i\) of flat 
and injective modules, respectively (over the same base).

\point[Modules over pro-finite algebras]
Let \(E\) be a pro-finite algebra over \(\kk\). Hence \(E\) is a co-filtering 
system of finite algebras, indexed by a category \(\cI\). Equivalently, it is 
given by an ind-object \(\spec(E)\) of the category \(\cS\) of finite schemes 
over \(\kk\) (i.e., it is a formal set in the terminology of~\Sec{sec:ms}).

Pulling back the categories \(\cC^f\) and \(\cC^i\) from above, we obtain 
fibred categories \(\Res{\cC^f}{\cI}\) and \(\Res{\cC^i}{\cI}\) over \(\cI\).

\begin{defn}\label{def:prolongp}
  For a pro-finite algebra \(E\), we defined \(\Flat{E}\), the category of 
  \Def{flat \(E\)-modules} in \(\cC\), to be the limit \(\Lim[\cI]{\cC^f}\) 
  of the fibred category of flat modules along \(\cI\). Dually, we define the 
  category of \(E\)-injective modules (or the \Def{\(E\)-prolongation of 
  \(\cC\)}) \(\Plng{E}\) as the limit \(\Lim[\cI]{\cC^i}\).
\end{defn}

We note that if \(E\) happens to be a finite algebra, this definition agrees 
with the previous one by the remarks following the definition of the limit.

\begin{cor}\label{cor:proflatop}
  For any pro-finite algebra \(E\), the categories \(\Flat{E}\) and 
  \(\Plng{E}\) are rigid (non-abelian) tensor categories, and 
  \(\oX\mt\Co{\oX}\) determines a tensor equivalence 
  \(\Plng{E}\ra\Op{\Flat{E}}\).
\end{cor}
\begin{proof}
  By Corollaries~\ref{cor:flatness},~\ref{cor:flatop} and~\ref{cor:rigid}, 
  and the discussions in~\ref{pt:limten} and~\ref{pt:limop}.
\end{proof}

We note that, as in the finite case, a flat coherent \(E\)-module \(M\) 
(i.e., an object of \(\Flat[\Vec]{E}\)) determines, for each object \(\oX\) 
of \(\cC\) an object \(M\Ten\oX\) of \(\Flat{E}\) and an object 
\(\Ehom(M,\oX)\) of \(\Plng{E}\).

\point\label{pt:pbpro}
The discussion on pullbacks (\ref{pt:pullbacks}) and 
Proposition~\ref{prp:pullbacks} extend to maps between pro-finite algebras.  
Let \(f:\spec(F)\ra\spec(E)\) correspond to a map between two pro-finite 
algebras \(E\) and \(F\). We define \(f^!:\Plng{E}\ra\Plng{F}\) as follows 
(\(f^*\) is analogous).

First, assume that \(F\) is finite. Then \(f\) is induced by some 
\(f_1:\spec(F)\ra\spec(E_1)\), where \(E_1\) is finite. Hence we have a 
functor \(f_1^!:\Plng{E_1}\ra\Plng{F}\). Given a Cartesian functor 
\(\oX:E\ra\cC^i\) (where we think of \(E\) as the index category), define 
\(f^!(\oX)=f_1^!(\oX(E_1))\). This is well defined since \(\oX\) is 
Cartesian.

A general \(F\) is the inverse limit of finite ones, and for each finite 
piece \(F_a\) we obtain from the previous step a functor 
\(f_a^!:\Plng{E}\ra\Plng{F_a}\). These functors form a matching family, and 
hence determine a functor to the limit \(\Plng{F}\).

\section{Fields with Operators}\label{sec:ms}
In this section, we recall the formalism introduced in~\Cite{MS1} 
and~\Cite{MS2} (adapted to our setting). As indicated in the introduction, 
this formalism provides a unified framework for fields endowed with 
operators, including differential and difference fields, and (Kolchin-style) 
algebraic geometry over them.

The use of geometric language is mainly for the purpose of intuition. The 
case we will eventually be interested in is when \(\sX=\spec(\kk_0)\) and 
\(\sZ=\spec(\kk)\), with \(\kk\) a field extending \(\kk_0\), and the reader 
will not lose (or gain) anything by assuming this from the beginning (on the 
other hand, some of the examples are mainly interesting when \(\kk_0\) is not 
a field).

We call a map \(f:\sX\ra\sY\) of schemes \Def{quasi-separable} if for any 
\(K\)-point \(y\in\sY(K)\), the fibre \(\sX_y\) of \(f\) over \(y\) has the 
form \(\spec(A)\), with \(A\) a quasi-separable \(K\)-algebra 
(Remark~\ref{rmk:nonperf}).  As with ramification, one could instead ask the 
same condition for schematic points, or for geometric points. We note that 
this property is preserved by base-change. This condition is not important 
for the current section, but will play a role later, for reasons explained 
in~\Sec{sec:prolong}.

\subsection{Formal sets}
\begin{defn}
  Let \(\sX\) be a quasi-compact Noetherian scheme. By a \Def{formal set} 
  over \(\sX\), we mean an ind-object in the category of flat, finite schemes 
  over \(\sX\) (i.e., schemes over \(\sX\) whose sheaf of algebras is flat 
  and coherent as an \(\sX\)-module).
  
  The formal set is \Def{strict} if it can be represented by a system of 
  closed embeddings. It is \Def{quasi-separable} if it can be represented by 
  a system of quasi-separable schemes over \(\sX\).
  
  A \Def{pointed} formal set is a formal set \(\FY\) together with a map 
  \(\sX\ra\FY\) over \(\sX\).

  A \Def{formal (abelian) monoid} (etc.) over \(\sX\) is an (abelian) monoid 
  object in this category.
\end{defn}

\begin{example}
  Let \(\set{S}\) be a finite set. The co-product 
  \(\set{S}\x\sX=\coprod_\set{S}\sX\) is clearly flat and finite over 
  \(\sX\), and is therefore a (finite) formal set over \(\sX\) (it is also 
  quasi-separable).  A map \(f:\set{S}\ra\set{T}\) to another finite set 
  induces a morphism \(f\x{}1:\set{S}\x\sX\ra\set{T}\x\sX\) over \(\sX\), and 
  any morphism over \(\sX\) comes from a map of sets. Hence 
  \(\set{S}\mt\set{S}\x\sX\) is a fully faithful exact embedding of finite 
  sets in formal sets over \(\sX\).  Since a set is the same as an ind-finite 
  set, this also determines a fully faithful exact embedding of the category 
  of sets into (strict, quasi-separable) formal sets.

  Since the embedding is exact, it induces an embedding of pointed sets, 
  (abelian) monoids, etc., into the category of formal such objects. For 
  instance, the monoids of natural numbers or integers can be viewed as 
  formal monoids over \(\sX\). As with usual sets, it is \emph{not} the case 
  that a formal monoid is an ind-object in finite formal monoids.

  As explained in the introduction, the guiding principle for all that 
  follows is that any construction or result valid for usual sets should 
  extend to formal sets.
\end{example}

\begin{example}
  Let \(\sY\) be a scheme over \(\sX\). If \(\sX\) is smooth of dimension at 
  most \(1\)\footnote{The condition on \(\sX\) is used to ensure that the 
  system is filtering, since in this case a module is flat if and only if it 
  is torsion-free, so the union of two finite flat subschemes is again flat 
  (and finite). It is possible that this condition is redundant, and the 
  construction can be carried out in general}, the collection of flat and 
  finite closed subschemes of \(\sY\) over \(\sX\) forms a filtering system, 
  and so determines a strict formal set \(\FY\) (which might be empty).
  
  There is a map \(\FY\ra\sY\) (in the category of ind-schemes over \(\sX\)), 
  and every map from a formal set to \(\sY\) (all over \(\sX\)) factors 
  uniquely via \(\FY\). For such \(\sX\), the previous example is obtained 
  from this one as a special case by taking, for an arbitrary set 
  \(\set{S}\), \(\sY=\coprod_\set{S}\sX\).
\end{example}

\begin{example}
  In the situation of the previous example, let \(\sY_0\) by a fixed 
  subscheme of \(\sY\), flat and finite over \(\sX\). The sub-system of 
  \(\FY\) consisting of subschemes with the same set-theoretic support as 
  \(\sY_0\) is again filtering (even without the restrictions on \(\sX\)), 
  and can be identified with the completion of \(\sY\) along \(\sY_0\).

  For instance, if \(\sX=\spec(\kk)\) is a point, \(\sY=\AA^1\) and \(\sY_0\) 
  the point at the origin, we obtain the formal scheme with one point, and 
  structure sheaf \(\kk[[x]]\).
\end{example}

\begin{example}\label{ex:fgl}
  As a special case of the previous example, any formal group law over a ring 
  \(\kk\) determines a (quasi-separable) formal abelian group over 
  \(\sX=\spec(\kk)\) in our sense.
\end{example}

\begin{example}\label{ex:hss}
  Any Hasse--Schmidt system \(\ul{\mathcal{D}}\) over \(A\), in the sense 
  of~\Cite[Def.~2.2]{MS2}, determines a pointed strict formal set 
  \((\spec(\mathcal{D}_i(A)))_i\) over \(\sX=\spec(A)\), indexed by the  
  natural numbers. Conversely, any such pointed strict formal set \((\sY_i)\) 
  can be extended to a Hasse--Schmidt system by choosing a compatible system 
  of \(A\) bases, where \(\mathcal{D}_i(\sZ)=\sY_i\x_\sX\sZ\) 
  (cf.~\Cite[after Def.~2.1]{MS2}).

  Likewise, an iterative Hasse--Schmidt system over \(A\) 
  (\Cite[Def.~2.17]{MS2}) is naturally identified (up to a choice of basis) 
  with a formal abelian monoid over \(\sX\).
\end{example}

We work in the category of ind-schemes over \(\sX\). Given a formal set 
\(\FY\), we get, for any scheme \(\sZ\) over \(\sX\), an ind-scheme 
\(\FY\x_\sX\sZ\), and if \(\FY\) is pointed, a map \(\sZ\ra\FY\x_\sX\sZ\).  
This process extends to the case of ind-schemes \(\sZ\).

\begin{defn}
  Let \(\FY\) be a pointed formal set over \(\sX\) (as above). A 
  \Def{\(\FY\)-scheme} is a scheme \(\sZ\) over \(\sX\), together with a map 
  \(\FY\x_\sX\sZ\ra\sZ\) over \(\sX\), such that the induced map 
  \(\sZ\ra\sZ\) resulting from the point is the identity.

  If \(\FY\) is a formal abelian monoid, an~\Def{iterative \(\FY\)-scheme} is 
  an \(\FY\)-scheme in which the structure map is a monoid action (when 
  \(\FY\) is such a monoid, all \(\FY\)-schemes we will consider will be 
  iterative, so we will usually omit the title ``iterative'').

  If \(\sZ\) is an (iterative) \(\FY\)-scheme, an \Def{(iterative) 
  \(\FY\)-scheme over \(\sZ\)} is a map of (iterative) \(\FY\)-schemes 
  \(\sW\ra\sZ\).
\end{defn}

\begin{example}
  Under the identification in Example~\ref{ex:hss}, an affine \(\FY\)-scheme 
  is the same as a Hasse--Schmidt ring in the sense of~\Cite[Def.~2.4]{MS2} 
  (note that the choice of basis in the definition of a Hasse--Schmidt system 
  does not play any role in the definition of a Hasse--Schmidt ring).  
  Likewise, with the adjective ``iterative'' added.
\end{example}

\begin{example}\label{ex:dual}
  If \(\sX\) is over \(\spec(\kk)\), and 
  \(\FY=\sX\x_{\spec(\kk)}\spec(\kk[\epsilon])\), where \(\kk[\epsilon]\) is 
  the ring of dual numbers (\(\epsilon^2=0\)), pointed in the only possible 
  way, then an \(\FY\)-structure on \(\sZ\) is the same as a vector field on 
  \(\sZ\) over \(\sX\).

  If \(\kk\) has characteristic \(2\), then \(\FY\) is a finite flat group 
  sub-scheme of the additive group over \(\sX\) (and therefore a formal group 
  in our sense), and an iterative \(\FY\)-structure on \(\sZ\) corresponds to 
  a vector field \(\dd\) such that \(\dd^2=0\).
\end{example}

\begin{example}\label{ex:deriv}
  If \(\sX\) is over \(\spec(\kk)\), where \(\kk\) is a field of 
  characteristic \(0\), and \(\FY=\sX[[t]]\), the additive formal group over 
  \(\sX\) (as in Example~\ref{ex:fgl}), then an (iterative) \(\FY\)-structure 
  on a scheme \(\sZ\) over \(\sX\) is again the same as a vector field on 
  \(\sZ\) over \(\sX\). In general, such a structure corresponds to a system 
  of Hasse--Schmidt derivations. This is explained in detail 
  in~\Cite[Prop.~2.20]{MS2}, but we recall the computation for convenience.  
  
  Since everything is local, we may assume that \(\sX=\spec(A)\) and 
  \(\sZ=\spec(B)\) are affine. Then \(\FY=\spec(A[[t]])\), and an 
  \(\FY\)-structure on \(\sZ\) is an algebra map \(d:B\ra{}B[[t]]\) over 
  \(A\). Hence it is given by \(d(b)=\sum_{i\in\w}\dd_i(b)t^i\) for some maps 
  \(\dd_i:B\ra{}B\). The statement that \(d\) is an algebra map means that 
  each \(\dd_i\) is an \(A\)-module map, and that 
  \(\dd_n(ab)=\sum_{i\le{}n}\dd_i(a)\dd_{n-i}(b)\) for \(a,b\in{}B\). The 
  condition on the base point \(A[[t]]\ra{}A\) means that \(\dd_0\) is the 
  identity.  Finally, iterativity means that \(d\) makes \(B\) an 
  \(A[[t]]\)-comodule (for the additive group law 
  \(c:t\mt{}t\Ten{}1+1\Ten{}t\)), so that for all \(b\in{}B\)
  \begin{equation}
    \sum_{i,j\in\w}\binom{i+j}{i}\dd_{i+j}(b)t^i\Ten t^j= (1\Ten{}c)(d(b))= 
    d(d(b))= \sum_{i,j\in\w}\dd_i(\dd_j(b))t^i\Ten{}t^j
  \end{equation}
  Hence, \(\binom{i+j}{i}\dd_{i+j}=\dd_i\circ\dd_j\), which is precisely the 
  definition of an iterative Hasse--Schmidt derivation. See also 
  Example~\ref{ex:freeadd}, where we make a similar computation for usual 
  derivations.
\end{example}

\begin{example}
  If \(\FY\) is a pointed set (over \(\sX\)), then an \(\FY\)-structure on 
  \(\sZ\) is simply a collection of endomorphisms of \(\sZ\), indexed by 
  \(\FY\), such that the point corresponds to the identity. Likewise, if 
  \(\FY\) is a (discrete) monoid (or group), an (iterative) \(\FY\)-structure 
  is an action of \(\FY\) on \(\sZ\). In particular, for \(\FY=(\ZZ,+)\), an 
  \(\FY\)-structure on \(\sZ\) is the same as an automorphism of \(\sZ\) 
  (cf.~\Cite[5.1]{MS2}).
\end{example}

\point[Free formal monoids]
We will see in Prop.~\ref{prp:free} below that to any pointed formal set 
\(\FY\) we may associate a free formal monoid \(\Free{\FY}\), such that 
\(\FY\)-schemes are identified with iterative \(\Free{\FY}\)-schemes. For 
this reason, we are free to restrict our attention to the iterative case from 
now on.

\subsection{Prolongations}\label{ssc:prolong}
From now on we fix a base scheme \(\sX\), and take all schemes, formal sets, 
etc., to be over \(\sX\), without mentioning it. We also fix a pointed formal 
set \(\FY_0\).

\point
For a scheme \(\sY\), we write \(\Sch[\sY]\) for the category of 
quasi-projective schemes over \(\sY\).

We say that a map \(q:\sW\ra\sY\) of ind-objects is \Def{compact} if for any 
map \(r:\sY_0\ra\sY\) where \(\sY_0\) is compact (i.e., an object in the 
original category), the pullback \(r^*(\sW)\) is compact as well (the term 
``proper'' would be better, but this becomes confusing in the context of 
schemes).  For \(\sY\) an ind-scheme, we denote by \(\Sch[\sY]\) the category 
of ind-quasi-projective schemes compact over \(\sY\) (note that a compact 
ind-scheme over a scheme is a scheme, so there is no contradiction).

\point[Weil restriction]
A map \(p:\sY\ra\sZ\) of schemes determines a base change functor 
\(p^*:\Sch[\sZ]\ra\Sch[\sY]\). When \(p\) is flat and finite, this functor 
has a right adjoint, \(p_*\), called the \Def{Weil restriction} (cf, e.g., 
\Cite[A.5]{restrict}). Hence, if \(q:\sW\ra\sY\) is a scheme over \(\sY\), 
and \(\sT\) is a scheme over \(\sZ\), a \(\sT\) point of \(p_*(\sW)\) 
corresponds to a family of sections of \(q\), parametrised by \(\sT\).

Given a diagram of schemes over \(\sZ\)
\begin{equation}
  \xymatrix{
  \sW_1:=r^*(\sW_2)\ar[rr]\ar_{q_1}[d]&&\sW_2\ar^{q_2}[d]\\
  \sY_1\ar_{p_1}[dr]\ar_r[rr]&&\sY_2\ar^{p_2}[dl]\\
  &\sZ&
  }
\end{equation}
a section of \(q_2\) restricts to a section of \(q_1\), so we obtain a map 
\({p_2}_*(\sW_2)\ra{p_1}_*(\sW_1)\) over \(\sZ\). Hence, if \(p:\FY\ra\sZ\) 
is a formal set over \(\sZ\), and \(\sW\) is a compact (quasi-projective) 
ind-scheme over \(\FY\), we obtain a pro-scheme \(p_*(\sW)\) over \(\sZ\). We 
note that if \(\sT\) is a scheme over \(\sZ\), the pullback \(p^*(\sT)\) is 
compact over \(\FY\), and we still have the adjunction property 
\(\Ehom_{\Sch[\FY]}(p^*(\sT),\sW)=\Ehom_{\Pro{\Sch[\sZ]}}(\sT,p_*(\sW))\).  This 
fact further extends formally to the case when \(\sT\) is an ind-scheme.

We note also that when the map \(r\) above is a closed embedding, the 
resulting map \({p_2}_*\sW_2\ra{p_1}_*\sW_1\) is dominant (since an open 
subset of \(\sW_1\) comes from an open subset of \(\sW_2\)), so when \(\FY\) 
is strict, \(p_*\sW\) is strict as well.

\begin{example}\label{ex:arc}
  If \(\kk\) is a field, \(\sZ=\spec(\kk)\), \(\FY=\spec(\kk[[x]])\), and 
  \(\sW\) is a scheme over \(\kk\), then \(p_*(p^*(\sW))\) is the arc space 
  of \(\sW\). The adjunction map \(\sW\ra{}p_*(p^*(\sW))\) is the 
  \(0\)-section.
\end{example}

\begin{remark}
  The notation is chosen so that it is compatible when the scheme \(\sW\) 
  over \(\sZ\) is identified with the presheaf (on \(\Sch[\sZ]\)) it 
  represents. When \(\sW\) is affine over \(\sZ\), it corresponds to an 
  \(\OO_\sZ\)-algebra (in particular, \(\OO_\sZ\)-module) \(\OO_\sW\), but 
  the operations above \emph{do not} correspond to the similarly denoted 
  operations on modules.
\end{remark}

\begin{defn}
  Let \(\mu:\FY_0\x\sZ\ra\sZ\) be an \(\FY_0\)-scheme, and let \(\sW\) be a 
  quasi-projective scheme over \(\sZ\). The \Def{\(\sZ\)-prolongation} of 
  \(\sW\) is the pro-scheme \(\tau(\sW)=p_*(\mu^*(\sW))\) over \(\sZ\), where 
  \(p\) is the projection \(p:\FY:=\FY_0\x\sZ\ra\sZ\). The base point of 
  \(\FY_0\) determines a map \(\pi^\tau_\sW:\tau(\sW)\ra\sW\).
\end{defn}

\point\label{pt:plngadj}
Thus, \(\tau(\sW)\) represents the functor \(\sT\mt\sW(\mu_!(\FY_0\x\sT))\) 
on \(\Sch[\sZ]\), where \(\mu_!(\FY_0\x\sT)\) is the ind-scheme 
\(\FY_0\x\sT\), with \(\sZ\) structure given by composition with \(\mu\) (in 
other words, \(\mu_!\) is the left adjoint of \(\mu^*\)).

A map \(\sT\ra\sZ\) of \(\FY_0\)-schemes induces a map 
\(\mu_!(\FY_0\x\sT)\ra\sT\), hence induces by the previous paragraph a 
function of sets \(\nabla^\sT:\sW(\sT)\ra\tau_\sZ(\sW)(\sT)\) 
(cf.~\Cite[Def.~2.10]{MS2}). In particular, an \(\FY_0\)-structure on \(\sW\) 
is the same as a section of \(\pi:\tau(\sW)\ra\sW\).

\point\label{pt:prolong}
The functor \(\tau\) extends in an obvious way to a functor on 
(quasi-projective) pro-schemes over \(\sZ\). In particular, \(\tau^2\sW\) 
makes sense. Assume now that we are given a monoid structure 
\(m:\FY_0\x\FY_0\ra\FY_0\), and that \(\mu\) is a monoid action. Then for any 
scheme \(\sT\) over \(\sZ\) we obtain a map 
\(m\x1_\sT:\mu_!(\FY_0\x\mu_!(\FY_0\x\sT))\ra\mu_!(\FY_0\x\sT)\). By taking 
\(\sT\)-points, this map induces a map \(m^\tau:\tau\ra\tau^2\), and the 
monoid axioms imply that \((\tau,\pi^\tau,m^\tau)\) is a co-monad (on 
\(\Pro{\Sch[\sZ]}\)). An \(\FY_0\)-scheme over \(\sZ\) is the same as a 
co-action of this co-monad (See~\Cite[\S~VI]{CWM} for co-monads; there is 
some more discussion on this in~\Sec{q:monads}).

In particular, each \(\tau\sW\) is an \(\FY_0\)-scheme over \(\sZ\), which is 
universal among \(\FY_0\)-schemes over \(\sW\). In other words, the functor 
\(\tau\) from \(\Pro{\Sch[\sZ]}\) to \(\FY_0\)-pro-schemes over \(\sZ\) is 
right adjoint to the forgetful functor (as promised in the introduction).

\begin{example}\label{ex:twisted}
  Let \(\kk\) be a field, \(\sZ=\spec(\kk)\) and \(\FY_0=\spec(\ZZ[\eps])\), 
  as in Example~\ref{ex:dual}, so that an \(\FY_0\)-structure on \(\sZ\) 
  corresponds to a derivation \('\) on \(\kk\). If \(\sW\) is a scheme over 
  \(\kk\), \(\tau(\sW)\) is then a scheme whose \(A\) points (for a 
  \(\kk\)-algebra \(A\)) are \(\sW(A[\eps])\), where the \(\kk\) algebra 
  structure on \(A[\eps]\) is given by \(x\mt{}x+x'\eps\). Hence 
  \(\tau(\sW)\) is the twisted tangent bundle of \(\sW\) over \(\kk\).

  If \(A\) itself is endowed with a vector field \('\) compatible with the 
  one on \(\kk\) (i.e., with an \(\FY_0\)-structure over \(\sZ\)), then the 
  map \(\nabla^A\) above is induced by pre-composing with 
  \(\spec(A[\eps])\ra\spec(A)\), \(a\mt{}a+a'\eps\), i.e., by differentiating 
  the \(A\)-points of \(\sW\). In particular, a vector field on \(\sW\) 
  (extending that on \(\kk\)) is the same as a section of the twisted tangent 
  bundle.
\end{example}

\begin{example}
  Generalising Example~\ref{ex:arc}, we may consider the special case 
  \(\sX=\sZ=\spec(\kk)\), with the trivial action of (any) \(\FY_0\). Then 
  \(\tau\sW\) is the analogue of the arc space (or the tangent bundle) for 
  \(\FY=\FY_0\). In particular, an \(\FY\)-structure on \(\sW\) is the same 
  as a section of \(\tau\sW\ra\sW\) (this is analogous to the statement that 
  a derivation on \(\sW\) is the same as a morphism \(\Omega^1\sW\ra\sW\) of 
  \(\sW\) modules).
  
  The functor \(\tau\sW\) is the internal-hom \(\Hom(\FY,\sW)\), in the sense 
  that \(\tau\sW(\sT)=\Ehom(\FY\x\sT,\sW)\) (a projective limit of morphisms 
  of schemes over \(\kk\)). This all remains true for any affine \(\sW\) (not 
  necessarily finitely generated), since any affine scheme can be viewed as 
  an inverse system of finite generated ones.

\end{example}

\point
We call a map of pro-schemes \(f:\sU\ra\sV\) a closed embedding if for any 
map \(p:\sU\ra\sU_0\), with \(\sU_0\) a scheme, there is a map 
\(q:\sV\ra\sV_0\) with \(\sV_0\) a scheme, such that \(q\circ{}f\) factors 
through \(f_0\circ{}p\), with \(f_0:\sU_0\ra\sV_0\) a closed embedding.

\begin{defn}
  Let \(\FY\) be a formal monoid acting on a scheme \(\sZ\), and let \(\sW\) 
  be a quasi-projective scheme over \(\sZ\). A \Def{\(\FY\)-subscheme} of 
  \(\sW\) is a closed subscheme of \(\tau\sW\) that is closed under the 
  action. In other words, it is an \(\FY\) pro-scheme \(\sW_1\) over \(\sW\), 
  such that the induced map \(\sW_1\ra\tau\sW\) is a closed embedding.
\end{defn}

\subsection{The affine picture}
Through most of this section, we only talk about formal sets, and not their 
actions, so we will use \(E\) and \(\FY\) in place of \(E_0\) and \(\FY_0\).

\point\label{pt:affine}
Assume that \(\sX=\spec(A)\). Then a formal set \(\FY\) over \(\sX\) 
corresponds to a projective system \(E=(E_i)\) of finite flat \(A\)-algebras, 
and a base point corresponds to an \(A\)-algebra map \(E\ra{}A\). A monoid 
structure \(m\) on \(\FY\) determines a bi-algebra structure 
\(m^*:E\ra{}E\Ten_A{}E\) (i.e., \(m^*\) is a map of pro-\(A\)-algebras. It is 
not, in general, induced from the finite levels).

Likewise, an affine \(\FY\)-scheme corresponds to an \(A\)-algebra \(B\), 
together with a pro-algebra map \(B\ra{}E\Ten_A{}B\) (inducing the identity 
when composed with the base point \(E\ra{}A\)). The map is iterative if it 
makes \(B\) a co-module over \(E\).

\begin{prop}\label{prp:free}
  Let \(\FY\) be a pointed formal set. Then there is a universal map 
  \(\FY\ra\Free{\FY}\) of pointed formal sets, where \(\Free{\FY}\) is a 
  formal monoid. This map identifies iterative \(\Free{\FY}\)-schemes with 
  \(\FY\)-schemes. Likewise, there is a free formal abelian monoid 
  \(\Ab{\FY}\).
\end{prop}
\begin{proof}
  It is enough to give an affine construction that localises, since the 
  universal property will ensure the glueing. With notation as above, we 
  first ignore the algebra structure, and view \(E\) as a pro-finite flat 
  \(A\)-module, together with an \(A\)-module map \(p:E\ra{}A\). We produce a 
  co-algebra \(TE\), and a universal map (from a co-algebra) \(\pi:TE\ra{}E\) 
  over \(A\). The construction is dual to that of the tensor algebra.
  
  Let \(E^{\Ten{}n}\) be the \(n\)-fold tensor power of \(E\) over \(A\), and 
  let \(E_n\subseteq{}E^{\Ten{}n}\) be the equaliser of all the maps 
  \(E^{\Ten{}n}\ra{}E^{\Ten{}n-1}\) obtained by tensoring \(p\) with identity 
  maps. \(E_n\) is finite (since \(A\) is Noetherian) and flat over \(A\) 
  (for example, if \(E\) is free, then so is \(E_n\), and the construction 
  localises). The unique map \(E_n\ra{}E^{\Ten{}n-1}\) determined by these 
  maps clearly factors through \(E_{n-1}\), and we set \(TE=(E_i)\).  We let 
  \(\pi\) be the projection on \(E_1=E\). The co-multiplication is given by 
  the map \(E_{i+j}\ra{}E_i\Ten_A{}E_j\) which is the restriction of the 
  identity map. It is clear that this is a co-algebra. To get the (co-) 
  commutative version, simply symmetrise the tensors.

  Given another co-algebra \(H\) and a map \(t:H\ra{}E\) over \(A\), we lift 
  it to a map \(t_n:H\ra{}E_n\) via 
  \(t^{\Ten{}n}\circ{}c^{\circ{}n-1}:H\ra{}E^{\Ten{}n}\), where 
  \(c^{\circ{}n-1}:H\ra{}H^{\Ten{}n}\) is the application of the 
  co-multiplication \(c\) of \(H\) \(n-1\) times. The co-algebra axioms imply 
  that this map factors through \(E_n\), and it is clearly a unique 
  co-algebra map over \(E\). We note also that \(T\) commutes with filtered 
  inverse limits (in pro-finite flat \(A\)-co-algebras). In particular, if 
  \(E\) is given by a system \((E^\alpha)\), then \(TE\) is the inverse limit 
  of the \(TE^\alpha\).

  Finally, assume that \(E\) is a system of algebras. By the remark above, we 
  may assume that \(E\) itself is a finite flat \(A\)-algebra. The 
  multiplication map \(m\) determines a map 
  \(m\circ\pi\Ten\pi:TE\Ten_A{}TE\ra{}E\) over \(A\), hence a co-algebra map 
  \(TE\Ten_A{}TE\ra{}TE\), which is easily seen to be an algebra map.
\end{proof}

\begin{example}
  If \(\FY\) is a discrete set, then the free monoid generated by it 
  coincides with the usual free monoid in the category of sets.
\end{example}

\begin{example}\label{ex:freeadd}
  If \(\FY=\spec(A[\eps])\), as in Example~\ref{ex:dual}, the bi-algebra of 
  the free monoid can be described as follows. Let 
  \(A[[\eps_1,\eps_2,\dots]]\) be the formal power series algebra in 
  countably many variables \(\eps_i\), each satisfying \(\eps_i^2=0\). The 
  symmetric group \(S_\w\) of the natural numbers acts on this algebra, and 
  \(TE\) is the sub-algebra of invariant elements (i.e., symmetric power 
  series). Each element of \(TE\) can be written as \(\sum_{i\in\w}a_ie_i\), 
  where \(e_i\) is the \(i\)-th elementary symmetric power series 
  \(e_i=\sum_{j_1<\dots<j_i}\eps_{j_1}\dots\eps_{j_i}\) in the variables 
  \(\eps_k\). In this presentation, the algebra structure is given by 
  \(e_ie_j=\binom{i+j}{i}e_{i+j}\), and the co-algebra structure is given by 
  \(e_k\mt\sum_{i\le{}k}e_i\Ten{}e_{k-i}\).

  We show, by explicit calculation, that \(TE\) satisfies the required 
  property (this can be contrasted with the calculation in 
  Example~\ref{ex:deriv}).  A map \(d:A\ra{}TE\) may thus be written as 
  \(d(a)=\dd_0(a)+\dd_1(a)e_1+\dots\), where \(\dd_i\) are some maps 
  \(A\ra{}A\). The requirement that the unit (corresponding to the map 
  \(e_i\mt{}0\) for \(i>0\)) acts as the identity means that \(\dd_0(a)=a\).  
  The requirement that \(d\) makes \(A\) a comodule (i.e., that we have a 
  monoid action) means the following: applying \(d\) again to the 
  coefficients of \(d(a)\), we obtain
  \begin{equation}
    d(d(a))=\sum_{i,j\in\w}\dd_i(\dd_j(a))e_i\Ten{}e_j
  \end{equation}
  (recall that we view \(TE\) as a pro-algebra, so the tensor product 
  consists of ``power series'' in the \(e_i\Ten{}e_j\)). Comparing this with 
  the co-algebra structure, we see that \(\dd_{i+j}(a)=\dd_i(\dd_j(a))\).  
  Hence \(d\) is determined by \(\dd_1\). Finally, the statement that \(d\) 
  is an algebra map means that \(\dd_1\) is a derivation (and the product 
  formula makes it consistent for higher \(i\)). In other words, an action of 
  \(\FY\) is precisely the same as a derivation (in any characteristic!), so 
  \(\spec(TE)\) is indeed the free monoid on \(\spec(A[\eps])\).

  The algebra map \(A[[x]]\ra{}A[\eps]\) from the additive formal group 
  induces the bi-algebra map \(f:A[[x]]\ra{}TE\), \(f(x)=e_1\), and when 
  \(A\) contains \(\QQ\), this map is an isomorphism. On the other hand, if 
  \(A\) contains \(\FF_p\), then \(TE\) is generated (as a power series 
  algebra) by the \(e_{p^k}\), with \(e_i^p=0\) for \(i>0\).

  We note that \(e_k^d\) is divisible by \(d!\), and the assignment 
  \(e_k^{(d)}=e_k^d/d!\) determines a divided power structure on \(TE\) (with 
  respect to the ideal generated by all \(e_i\)), which could be called the 
  universal complete divided power \(A\)-algebra in one variable.
\end{example}

\point[Cartier duality]\label{sss:cartier}
The system \(E=(E_i)\) determines a direct system \(\Co{E}=(\Co{E_i})\) of 
finite dimensional co-algebras over \(A\) (\(\Co{E_i}\) is the dual with 
respect to \(A\)).  Since the category of co-algebras over \(A\) is 
equivalent (by taking limits) to the category of ind-finite co-algebras 
(\DM{Prop.~2.3}), this is just a co-algebra over \(A\). A base point 
\(E\ra{}A\) corresponds to an element \(1\in\Co{E}\). A monoid structure then 
corresponds to an algebra structure \(m_*:\Co{E}\Ten_A\Co{E}\ra\Co{E}\), 
which commutes with the co-algebra structure, and which is commutative if the 
original monoid was commutative.

Hence, to a commutative formal monoid \(\FY\) corresponds a commutative 
affine monoid scheme \(\Co{\FY}\), which we call the \Def{Cartier dual} of 
\(\FY\) (this is precisely the usual Cartier duality when \(\FY\) is finite, 
cf~\Cite[2.4]{waterhouse} or~\Cite{pinkfgs}).  Reversing the arguments above, 
we see that conversely, to a commutative affine monoid scheme over \(A\) 
corresponds a commutative formal monoid, and that the two operations are 
inverse to each other. We also note that the \(A\)-points of \(\FY\) can be 
viewed as elements of \(\Co{E}\).

Given an \(\FY\)-scheme corresponding to an \(A\)-algebra \(B\), the 
co-module structure on \(B\) corresponds to a module structure for 
\(\Co{E}\). The fact that the co-module structure is an algebra map means 
that we have the following commutative diagram:

\begin{equation}\label{eq:comod}
  \xymatrix{
  \Co{E}\Ten B\Ten B\ar^{c\Ten 1\Ten 1}[rr]\ar_{1\Ten m}[d] && 
  \Co{E}\Ten\Co{E}\Ten B\Ten B\ar[d]\\
  \Co{E}\Ten B\ar[r]&B&\ar^m[l]B\Ten B
  }
\end{equation}
where \(c\) is the co-multiplication of \(\Co{E}\), and \(m\) is the 
multiplication on \(B\).

\begin{example}\label{ex:cartiergm}
  Assume that \(\FY\) is a discrete commutative monoid \(Y\). Then \(\Co{E}\) 
  is the group algebra \(A[Y]\), with co-algebra structure given by 
  \(y\mt{}y\Ten{}y\) for \(y\in{}Y\). An \(A[Y]\) module is then the same as 
  an action of \(Y\) by \(A\)-linear map. The diagram~\eqref{eq:comod} then 
  means that \(Y\) acts by \(A\)-algebra endomorphisms, i.e., \(Y\) acts on 
  \(\spec(B)\). For 
\end{example}

\begin{example}\label{ex:cartierga}
  Let \(\FY\) be the additive formal group (\(E=A[[x]]\)). Then \(\Co{E}\) is 
  the \(A\)-algebra generated by elements \(u_i\), \(i>0\), with relations 
  \(u_iu_j=\binom{i+j}{i}u_{i+j}\), and co-multiplication 
  \(c(u_n)=\sum_{i\le{}n}u_i\Ten{}u_{n-i}\). In other words, it is the 
  sub-bi-algebra of the algebra of the free monoid on the dual numbers 
  consisting of finite sums (in yet other words, it is the free divided 
  powers algebra in one variable over \(A\)).

  If \(A\) has characteristic \(0\), we get \(\Co{\FY}=\sGa\), the additive 
  group. In particular, a module over \(\Co{E}=A[x]\) is simply an 
  \(A\)-linear action of \(x\). Diagram~\eqref{eq:comod} then reflects that 
  \(x\) is a derivation. Similarly, in characteristic \(p>0\), \(\Co{E}\) is 
  generated by \(u_{p^k}\) for \(k>0\) (with some relations), and an action 
  satisfying~\eqref{eq:comod} corresponds to a sequence of Hasse--Schmidt 
  derivations.
\end{example}

\begin{remark}
  The procedure described in~\ref{sss:cartier} is valid also when the monoid 
  is not commutative, but the resulting algebra \(\Co{E}\) is not 
  commutative, so the geometric interpretation as a scheme is no longer 
  available.
\end{remark}

\begin{remark}
  If \(\FZ\) is a formal monoid acting on \(\FY\) by monoid endomorphisms, 
  then it also acts on \(\Co{\FY}\), making \(\Co{\FY}\) a \(\FZ\)-scheme, on 
  which \(\FZ\) acts by monoid endomorphisms.

  This happens for example if \(\FY\) is (the additive monoid of) a formal 
  semi-ring, and \(\FZ\) is the multiplicative monoid. For instance, if 
  \(\FY\) is the (discrete) ring of integers, then \(\Co{\FY}\) is the 
  multiplicative group, and \(\ZZ\) acts by endomorphisms in the usual way.  
  Likewise, the dual of \(\ZZ[i]\) is \({\sGm}^2\), with \(i(a,b)=(-b,a)\).

  It also happens with the additive formal group, on which the usual 
  derivation acts by group endomorphisms. We therefore get a derivation on 
  the dual, the divided powers algebra, given by \(u_i'=u_{i-1}\).
\end{remark}

\point[The prolongations of affine spaces]
Assume again, as in~\ref{pt:affine}, that we are given a formal monoid 
\(\FY_0\) over an affine scheme \(\spec(A)\), acting on \(\sZ=\spec(\kk)\), 
where \(\kk\) is a field.  Then \(\FY=\FY_0\Ten_A\kk\) is given by a 
projective system \(E=(E_i)\) of finite algebras over \(\kk\). We denote the 
projection and the action maps \(\FY\ra\spec(\kk)\) by \(p\) and \(\mu\), 
respectively.  The correspond to pro-algebra maps \(\kk\ra{}E\) (over \(A\)).

Given a finite dimensional vector space \(V\) over \(\kk\), we let 
\(\Aff{V}=\spec(\Sym(\Co{V}))\) be the associated affine space. Hence, for 
any \(\kk\)-algebra \(B\), the \(B\)-points of \(\Aff{V}\) correspond to 
\(\kk\)-linear maps \(\Co{V}\ra{}B\), i.e., to elements of \(V\Ten_\kk{}B\).

More generally, for a projective system \(V=(V_i)\) of such spaces, 
\(\Aff{V}=(\Aff{V_i})\) is the corresponding pro-scheme. We would like to 
compute the prolongation \(\tau\Aff{V}\) with respect to the given action.  
We denote by \(E\Ten_\mu{}V\) the tensor product over \(\kk\), where \(E\) is 
given a \(\kk\)-structure via \(\mu\). We view it as a vector space over 
\(\kk\) via the map \(p\).

\begin{prop}\label{prp:plngaff}
  For a (pro-) finite-dimensional vector space \(V\) over \(\kk\), 
  \(\tau\Aff{V}=\Aff{E\Ten_\mu{}V}\).
\end{prop}
\begin{proof}
  It is enough to prove that for any \(\kk\)-algebra \(B\), the two 
  pro-schemes have the same \(B\)-points. Also, it suffices to prove the 
  statement when \(E\) and \(V\) are finite.
  
  By~\ref{pt:plngadj}, the \(B\)-points of \(\tau\Aff{V}\) correspond to the 
  \(B_E=B\Ten_\kk{}E\) points of \(\Aff{V}\), where the tensor product is 
  taken with respect to the \(\kk\)-vector space structure on \(E\) given by 
  \(p\), but the \(\kk\)-structure on \(B_E\) is given by \(\mu\). Hence, by 
  the above discussion, they correspond to elements of 
  \((B\Ten_\kk{}E)\Ten_\mu{}V=B\Ten_\kk(E\Ten_\mu{}V)\). Again by the same 
  discussion, these elements correspond to the \(B\)-points of 
  \(\Aff{E\Ten_\mu{}V}\).
\end{proof}
  
\begin{remark}
  It is easy to describe the action \(\FY\x\tau\Aff{V}\ra\tau\Aff{V}\) in 
  these terms: it suffices to give an (ind-pro-) vector space map 
  \(\Co{E}\Ten_\kk{}E\Ten_\mu{}V\ra{}E\Ten_\mu{}V\). The map is given by the 
  ``transpose'' \(m^t:\Co{E}\Ten_\kk{}E\ra{}E\) of the co-algebra map 
  \(m:E\ra{}E\Ten_\kk{}E\).
\end{remark}

\section{Tannakian Categories}\label{sec:main}
We now arrive at the main point, the description of the category of 
representations of a linear group. The description is completely analogous to 
the one given in~\Cite[4]{tannakian} in the special case of differential 
fields. However, the proof is simpler, since we reduce to the algebraic case, 
instead of mimicking its proof.

\subsection{Linear groups}\label{ssc:lingp}
We fix a base action \(\FY_0\x\sZ\ra\sZ\) with \(\sZ=\spec(\kk)\), \(\kk\) a 
field, and work in the category of \(\FY_0\)-pro-schemes over \(\sZ\) (as 
before, \(\FY_0\), \(\sZ\) and all maps, products, etc. are over some base 
ring \(\kk_0\), which we generally omit from the notation. \(\FY_0\) is 
assumed to be quasi-separable over \(\kk_0\)).  We set \(\FY=\FY_0\x\sZ\).  
As explained in~\ref{pt:prolong}, each scheme \(\sX\) over \(\sZ\) (in the 
usual sense) determines an \(\FY\)-pro-scheme \(\tau\sX\).  Since \(\tau\) 
has a left-adjoint, it preserves products. In particular, a group pro-scheme 
\(\sG\) over \(\sZ\) determines a group object \(\tau\sG\) in the category of 
\(\FY\)-pro-schemes over \(\sZ\) (we call these \Def{\(\FY\)-groups} from now 
on).

\begin{defn}
  Let \(\sG\) be an \(\FY\)-group. A \Def{representation} of \(\sG\) is a map 
  (of \(\FY\)-groups) \(\sG\ra\tau\GL(V)\) for some finite dimensional 
  \(\kk\)-vector space \(V\). As customary, we sometimes write \(V\) for the 
  whole representation. A representation is \Def{faithful} if it is a closed 
  embedding. The group \(\sG\) is \Def{linear} if it admit a faithful 
  representation.
\end{defn}

\point
We note that already in the differential case, there are affine groups that 
are not linear (\Cite{cassidy2}), so the definition is reasonable. We also 
note that we have a slight discrepancy with the terminology of~\DM{Cor.~2.5}.

Given an \(\FY\)-group scheme \(\sG\), we denote by \(\Alg{\sG}\) the 
underlying group-pro-scheme. If \(\sG\) is a linear \(\FY\)-group, the 
category \(\Rep_\sG\) of representations of \(\sG\) is abelian and 
\(\kk\)-linear in the usual way.  With the usual tensor structure, it is a 
rigid tensor category. The forgetful functor shows it is neutral Tannakian.  
We have the following simple observation.

\begin{prop}\label{prp:underly}
  Let \(\sG\) be a linear \(\FY\)-group. The algebraic group associated to 
  the Tannakian category \(\Rep_\sG\) is \(\Alg{\sG}\).
\end{prop}
\begin{proof}
  By~\ref{pt:prolong}, \(\tau\) is right adjoint to the forgetful functor. 
  Since all functors involved are left exact, we get a similar result for 
  groups. Applying this to the map \(\sG\ra\tau\GL(V)\), we get the result.
\end{proof}

\point
Our goal is thus to describe an additional structure on \(\cC=\Rep_\sG\) that 
will allow us to recover the action of \(\FY\). We pass back to algebra: let 
\(E\) be the pro-algebra corresponding to \(\FY\x\sZ\). We ignore, at first, 
the monoid structure on \(\FY\), and so deal with each piece separately.  
Thus, we assume that \(E\) is a finite algebra. The projection and action 
maps are denoted by \(p,\mu:\spec(E)\ra\spec(\kk)\), respectively.

\point
Recall from~\ref{pt:pullbacks}, that given a map \(f:\spec(E)\ra\spec(\kk)\), 
there is a pullback functor \(f^*:\cC\ra\Lmod{E}\), given by 
\(f^*(\oX)=E\Ten_\kk\oX\). We note that in the present situation, the functor 
is defined even if \(f\) is not finite. When \(f\) is finite, \(f^*\) has a 
right adjoint, \(f_*\) given by viewing an \(E\)-module as a \(\kk\)-vector 
space via \(f\) (in general \(f_*\) is defined as a functor into 
\(\Ind{\cC}\), but we will not need it).  \(f^*\) is a tensor functor, and we 
have an internal version of the adjunction:
\begin{equation}
  f_*(\Hom_E(f^*(\oX),\oY))=\Hom(\oX,f_*(\oY))
\end{equation}
for any object \(\oX\) and \(E\)-module \(\oY\). We also have an isomorphism
\begin{equation}
  f_*(f^*(\oX)\Ten_E\oY)=\oX\Ten_\kk{}f_*(\oY)
\end{equation}

Assume that \(f\) is finite. Then \(f_*\) also has a right adjoint 
\(f^!:\cC\ra\Lmod{E}\), given by \(f^!(\oX)=\Hom_\kk(E,\oX)\). We note that 
\(f^!(\1)=\Co{E}\), where duality is with respect to \(f\). More generally, 
\(\Co{M}=\Hom_E(M,f^!(\1))\) for a finite \(E\)-module \(M\), and we may 
prefer the second notation to stress the dependence on \(f\). We thus have, 
for any object \(\oX\),  an isomorphism (in \(\Lmod{E}\)), as in 
Proposition~\ref{prp:pullbacks}
\begin{equation}
  f^!(\Co{\oX})=\Hom_E(f^*(\oX),f^!(\1))=\Co{(f^*(\oX))}
\end{equation}

In particular, \(f^*(\oX)\) is \(E\)-flat and \(f^!(\oX)\) is \(E\)-injective 
(as \(E\)-modules, disregarding the action of \(\sG\)). Using the identities 
above, we again have an internal version of the adjunction:
\begin{equation}
  \Hom_\kk(f_*(\oX),\oY)=f_*(\Hom_E(\oX,f^!(\oY)))
\end{equation}

Viewing \(f^*\) as a functor to the category \(\Flat{E}\) of flat 
\(E\)-modules in \(\cC\), \(f^*\) also has a left adjoint, 
\(f_!:\Flat{E}\ra\cC\), given by
\begin{equation}
  f_!(\oX)=f_*(f^!(\1)\Ten_E\oX)=\Co{E}\Ten_E\oX
\end{equation}
(this is obviously true when \(\oX\) is free, hence when \(\oX\) is flat by 
localisation). We have, by definition,
\begin{equation}
  f_!f^*=f_*f^!
\end{equation}
We note \(f_!(\oX)\) has, in fact, the structure of an \(E\)-injective 
module.  We also note that dually, the functor \(f^!:\cC\ra\Plng{E}\) has a 
right adjoint \(f_\#:\Plng{E}\ra\cC\), defined by 
\(f_{\#}(\oX)=\Hom_E(f^!(\1),\oX)\), but we will not use it.

\point
We would like to apply the discussion above to the maps \(p\) and \(\mu\).  
Note that \(p\), but not necessarily \(\mu\), is finite. We will be 
interested in the functor
\begin{equation}
  \tau(\oX)=p_!(\mu^*(\oX))=\Co{E}\Ten_\mu\oX
\end{equation}
which we view as a functor into either \(\cC\) or \(\Plng{E}\).  Our interest 
in this functor is explained by Proposition~\ref{prp:plngaff}, and the 
following fact.

\begin{lemma}\label{lem:plngbase}
  Let \(\oX\) be a representation, and let \(M\) be an \(E\)-module.
  \begin{enumerate}
    \item There is a canonical isomorphism 
      \(\Co{\tau(\oX)}=p_*\mu^*(\Co{\oX})\) as \(E\)-modules.
    \item If \(M\) is a flat (or injective) \(E\)-module, then 
      \(M\Ten_E{}\mu^*(\oX)\) is \(E\)-flat (\(E\)-injective) in 
      \(\Rep_\sG\).
  \end{enumerate}
\end{lemma}
\begin{proof}
  \begin{enumerate}
    \item We claim that both sides are isomorphic (as \(E\)-modules) to the 
      space \(\Hom_E(\mu^*(\oX),E)\). For the left side, this follows 
      directly from the adjunction. For the right side, let 
      \(\phi\in\Co{\oX}\). Then \(\mu^\#\circ\phi\) is a map from \(\oX\) to 
      \(E\), linear with the respect to the vector-space structure on \(E\) 
      given by \(\mu^\#\) (the algebra map corresponding to \(\mu\)). This is 
      equivalent to a map \(\mu^*(\oX)\ra{}E\) of \(E\)-modules, so we have a 
      map \(\Co{\oX}\ra\Hom_E(\mu^*(\oX),E)\), which is again 
      \(\mu^\#\)-linear.  We obtain an \(E\)-module map 
      \(\mu^*(\Co{\oX})\ra\Hom_E(\mu^*(\oX),E)\), which is an isomorphism by 
      dimension.
    \item This is clear, since \(M\Ten_E\mu^*(\oX)\) is a (finite) direct sum 
      of copies of \(M\).\qedhere
  \end{enumerate}
\end{proof}

\begin{remark}\label{rmk:dualc}
  As a result of this Lemma, we could, instead, work with the functor 
  \(\oX\mt{}p_*\mu^*\oX\), which is a tensor functor into \(\Flat{E}\), and 
  is perhaps more familiar. We choose to use the current setting mostly since 
  it is compatible with the original setup of~\Cite[4]{tannakian}, and also 
  because in our current setting, \(\tau(\Co{\oX})\) has a simple 
  interpretation as consisting of functions on \(\oX\) (as in 
  Proposition~\ref{prp:plngaff}). We discuss this again in the abstract 
  setting of the next section, in Remark~\ref{rmk:dual}.
\end{remark}

We now describe the properties of the functor \(\tau\). Eventually, we will 
use these properties to characterise the situation.

\begin{prop}
  The functor \(\tau\) is naturally a tensor functor from \(\Rep_\sG\) to 
  \(\Plng[\Rep_\sG]{E}\).
\end{prop}

We note that \(\tau\) is not \(\kk\)-linear.

\begin{proof}
  The fact that \(\tau\) takes values in \(\Plng[\Rep_\sG]{E}\) is explained 
  above. To give \(\tau\) a tensor structure, we need to provide functorial 
  (\(E\)-module) isomorphisms 
  \(\tau(U\Ten_\kk{}V)=\Hom_E(\Co{\tau(U)},\tau(V))\) 
  (Proposition~\ref{prp:modten}). The left hand side is isomorphic to 
  \(\mu^*(U)\Ten_E\tau(V)\) (directly from definitions), while by 
  Lemma~\ref{lem:plngbase}, the right hand side is isomorphic to
  \begin{equation}
    \Hom_E(\mu^*(\Co{U}),\tau(V))=\Hom_E(\mu^*(\Co{U}),E)\Ten_E\tau(V)=\mu^*(U)\Ten_E\tau(V)
  \end{equation}
  The verification that this is a tensor structure is straightforward.
\end{proof}

We now wish to change the algebra.

\begin{prop}\label{prp:algchange}
  Assume that \(E_1\) and \(E_2\) are two rings with maps \(p_1,\mu_1\) and 
  \(p_2,\mu_2\) as above, and corresponding functors \(\tau_1\) and 
  \(\tau_2\). Assume, further, that we are given a ring map \(f:E_1\ra{}E_2\) 
  that preserves both \(\kk\)-algebra structures. Then for any representation 
  \(V\) we have an isomorphism of \(E_2\)-modules 
  \(\Hom_{E_1}(E_2,\tau_1(V))=\tau_2(V)\), and together these isomorphisms 
  determine an isomorphism of tensor functors.
\end{prop}
\begin{proof}
  Using \(\tau_i(V)=\tau_i(\1)\Ten_{E_i}\mu_i^*(V)\) (as in the previous 
  proof), and \(\mu_2^*(V)=E_2\Ten_{E_1}\mu_1^*(V)\), we reduce to the case 
  \(V=\1\). Hence, we need to prove that 
  \(\Hom_{E_1}(E_2,\Co{E_1})=\Co{E_2}\) (as \(E_2\)-modules). But this is 
  obvious, by taking duals.
\end{proof}

\point
We recall that \(\FY\) was assumed to have a base point, which acts as the 
identity. In other words, we are also given a map 
\(i:\spec(\kk)\ra\spec(E)\), such that \(\mu\circ{}i=p\circ{}i\). The map 
\(i\) induces, as before, a functor \(i^!:\Plng{E}\ra\cC\), 
\(i^!(\oX)=\Ehom_E(\kk,\oX)\) (see also~\ref{pt:pullbacks}; geometrically, 
\(i^!(\oX)\) consists of sections of \(\oX\) supported at the base point).

The functor \(i^!\) extends, as in~\ref{pt:pbpro}, to \(\Plng{E}\) when \(E\) 
is a pro-finite algebra. Applying the previous proposition we obtain, upon 
passing to inverse systems, the following result.

\begin{cor}\label{cor:erep}
  Let \(\FY=\FY_0\x\sZ\) be a formal set (with \(\sZ=\spec(\kk)\)), let 
  \(p:\FY\ra\sZ\) be the projection, \(i:\sZ\ra\FY\) a base point (section of 
  \(p\)), and let \(\mu:\FY\ra\sZ\) be a map (action), such that 
  \(\mu\circ{}i\) is the identity. The definitions above determine a tensor 
  functor \(\tau:\cC=\Rep_\sG\ra\Plng{E}\), and a (tensor) isomorphism 
  \(i^*(\tau(V))=V\).
\end{cor}
\begin{proof}
  Apply the proposition above to the maps \(i:E_\alpha\ra{}E_2=\kk\) and the 
  transition maps \(E_\beta\ra{}E_\alpha\) of the system, using the 
  definition of \(\Plng{E}\) in~\ref{pt:pullbacks}.
\end{proof}

Finally, we bring back the monoid structure. Let 
\(m:\spec(E\Ten_\kk{}E)\ra\spec(E)\) be the product map. As 
in~\ref{pt:pbpro}, \(m\) determines the functor 
\(m^!:\Plng{E}\ra\Plng{E\Ten_\kk{}E}\). On the other hand, \(\tau\) extends 
to a functor on ind-objects of \(\cC\) (which we again denote \(\tau\)). We 
note that \(E\Ten_\kk{}E\) is isomorphic (as a \(\kk\)-algebra), to 
\(E\Ten_\mu{}E\), so \(\tau\circ\tau\) can be viewed as a functor to 
\(\Plng{E\Ten_\kk{}E}\). Proposition~\ref{prp:algchange} directly generalises 
to the case where the \(E_i\) are pro-finite algebras, and we obtain:

\begin{prop}\label{prp:erep}
  There is an \(E\Ten_\kk{}E\)-linear tensor isomorphism 
  \(\tau\circ\tau\ra{}m^!\circ\tau\).
\end{prop}
\begin{proof}
  The condition that \(\mu\) is an action means that 
  \((1\Ten\mu^\#)\circ\mu^\#=m^\#\circ\mu^\#:\kk\ra{}E\Ten_\kk{}E\) (where we 
  again identify \(E\Ten_\kk{}E\) with \(E\Ten_\mu{}E\)). In other words, 
  \(m^\#:E\ra{}E\Ten_\kk{}E\) maps the action \(\mu\) to the action 
  \(\mu\circ(1\x\mu)\). Applying Proposition~\ref{prp:algchange} with 
  \(\mu_1=\mu\), \(\mu_2=\mu\circ(1\x\mu)\) and \(f=m^\#\), we obtain an 
  \(E\Ten_\kk{}E\) isomorphism of \(m^!\circ\tau\) with the functor 
  \(\oX\mt{}\Co{(E\Ten_k{}E)}\Ten_\mu\oX\). This functor is the same as 
  \(\tau^2\), by Lemma~\ref{lem:plngbase}.
\end{proof}

\subsection{\(\FY\)-Tensor categories}
We now introduce the abstract axiomatisation of the situation described for 
representations. As usual, we fix a base ring \(\kk_0\), a field \(\kk\) over 
\(\kk_0\), and a quasi-separable formal monoid \((\FY_0,i_0,m_0)\) over 
\(\kk_0\).  We denote by \(E_0\) the \(\kk_0\)-pro-finite algebra 
corresponding to \(\FY_0\), and set \(E=E_0\Ten_{\kk_0}\kk\) and 
\(\FY=\spec(E)=\FY_0\x\sZ\).  As before, we denote by \(p:\FY\ra\sZ\) the 
projection.

\point
The base point \(i_0:\spec(\kk_0)\ra\FY_0\) and the product 
\(m_0:\FY_0\x\FY_0\ra\FY_0\) induce, by base change, maps \(i:\sZ\ra\FY\) and 
\(m:\FY\x_\sZ\FY\ra\FY\) over \(\kk\). Given an abelian \(\kk\)-linear tensor 
category \(\cC\), these maps determine functors \(i^!:\Plng{E}\ra\cC\) and 
\(m^!:\Plng{E}\ra\Plng{E\Ten_\kk{}E}\), as in~\ref{pt:pbpro}.

We note that \(\Plng{E\Ten_\kk{}E}\) can be viewed as a full subcategory of 
\(\Lmod{(E\Ten_\kk{}E)}=(\Lmod{E})\Ten_\kk(\Lmod{E})\). The fact that \(i_0\) 
is the unit for the action translates into isomorphisms of 
\(i^!\Ten_\kk\1\circ{}m^!\) and \(\1\Ten_\kk{}i^!\circ{}m^!\) with the 
identity on \(\Plng{E}\), and similarly for the associativity of \(m\).

\point
If \(\cC\) and \(\cD\) are two categories as above, and \(\fF:\cC\ra\cD\) is 
an exact tensor functor, then \(\fF\) induces a functor from \(\Plng{E}\) to 
\(\Plng[\cD]{E}\), which we denote by \(\Plng[\fF]{E}\). We have a natural 
(tensor, \(\kk\)-linear) isomorphism \(i^!\circ\Plng[\fF]{E}=\fF\circ{}i^!\), 
which we indeed denote by ``\(=\)''.  Note that \(i^!\) denotes the 
corresponding functor in both categories.  Similar remarks apply to the other 
canonically determined functors: \(m^!\), \(p_*\), etc. As in the previous 
section, we will sometimes omit \(p_*\).

\point
Assume that \(\tau\) is a functor from \(\cC\) to \(\Plng{E}\). If some 
finite algebra \(F\) over \(\kk\) acts on an object \(\oX\) of \(\cC\), 
applying \(\tau\) we obtain an induced action of \(F\) on \(\tau(\oX)\).  
Hence, \(\tau(\oX)\) is an \(E\Ten{}F\)-module in \(\cC\) (the tensor product 
is over \(\kk_0\) if \(\tau\) is \(\kk_0\)-linear). If \(\tau\) is a tensor 
functor, it restricts to a \(\kk_0\)-algebra map
\begin{equation}
  \mu^\#:=\tau_\1:\kk=\Eend(\1)\ra{}\Eend(\tau(\1))=E
\end{equation}
and \(\tau(\oX)\) is then a \(E\Ten_\mu{}F\)-module. As before, 
\(E\Ten_\mu{}F\) (with \(\kk\)-structure coming from the action on \(E\)) is 
identified, as a \(\kk\)-algebra, with \(E\Ten_\kk{}F\). In particular, each 
\(\tau^2(\oX)\) is an \(E\Ten_\kk{}E\)-module.

\begin{defn}\label{def:etensor}
  With notation as above.
  \begin{enumerate}
    \item
      An \Def{\(E_0\)-structure} (or \(\FY_0\)-structure) on a \(\kk\)-linear 
      tensor category \(\cC\) consists of the following data:
      \begin{enumerate}
        \item A \(\kk_0\)-linear tensor functor \(\tau\) from \(\cC\) to 
          \(\Plng{E}\), which is exact when viewed as a functor into 
          \(\Lmod{E}\).
        \item A \(\kk\)-linear tensor isomorphism
          \begin{equation}\label{eq:etensor}
            \nat{a}:i^!\circ\tau\ra Id_\cC
          \end{equation}
        \item An \(E\Ten_\kk{}E\)-linear tensor isomorphism
          \begin{equation}
            \nat{b}:\tau\circ\tau\ra m^!\circ\tau
          \end{equation}
      \end{enumerate}
      An \Def{\(E_0\)-tensor category} is a \(\kk\)-linear tensor category 
      together with an \(E_0\)-structure.

    \item
      If \((\cC,\tau,\nat{a},\nat{b})\) and \((\cD,\sigma,\nat{c},\nat{d})\) 
      are \(E_0\)-tensor categories (where \(\cD\) is allowed to be over a 
      different field \(K\)), an \Def{\(E_0\)-functor} from the first to the 
      second consists of an exact \(\kk\)-linear tensor functor 
      \(\fF:\cC\ra\cD\), together with an \(E\)-linear tensor isomorphism 
      \(u\):
      \begin{equation}
        \xymatrix{
        \cC\ar[r]^{\fF}\ar[d]_{\tau} & \cD\ar[d]^{\sigma}\ar@{=>}[dl]_u\\
        \Plng{E}\ar[r]_{\Plng[\fF]{E}} & \Plng[\cD]{E}
        }
      \end{equation}

      This data is required to satisfy the obvious commutation relation with 
      the structure isomorphism: The diagrams
      \begin{gather}
        \xymatrix{
        i^!\sigma\fF(\oX)\ar[d]_{i^!(u_\oX)}\ar[r]^{c_{\fF(\oX)}} &
          \fF(\oX)\\
        i^!(\Plng[\fF]{E}(\tau\oX))\ar@{=}[r] &
          \fF(i^!\tau\oX)\ar[u]_{\fF(a_\oX)}
        }\\
        \intertext{and}
        \xymatrix{
        \sigma\sigma\fF(\oX)\ar[d]_{\sigma(u_\oX)}\ar[r]^{d_{\fF(\oX)}} &
          m^!\sigma(\fF(\oX))\ar[d]^{m^!(u_\oX)}\\
        \sigma(\fF(\tau\oX))\ar[d]_{u_{\tau\oX}} &
          m^!\fF(\tau\oX)\ar@{=}[d]\\
        \fF(\tau\tau\oX)\ar[r]_{\fF(b_\oX)} & \fF(m^!\tau\oX)
        }
      \end{gather}
      commute.

    \item
      If \((\fF,u)\) and \((\fG,v)\) are \(E_0\)-functors from \((\cC,\dots)\) 
      to \((\cD,\dots)\), an \Def{\(E_0\)-map} from \((\fF,u)\) to 
      \((\fG,v)\) is a (\(K\)-linear) map \(\nat{r}:\fF\ra\fG\) of tensor 
      functors, such that the diagram
      \begin{equation}
        \xymatrix{
        \sigma\fF(\oX)\ar[r]^{u_\oX}\ar[d]_{\sigma(r_\oX)} &
          \fF(\tau\oX)\ar[d]^{r_{\tau\oX}}\\
        \sigma\fG(\oX)\ar[r]_{v_\oX} & \fG(\sigma\oX)
        }
      \end{equation}
      commutes.
  \end{enumerate}
\end{defn}

\begin{remark}\label{rmk:dual}
  A functor \(\tau\) as in the definition determines a functor 
  \(\Co{\tau}:\oX\mt\Co{(\tau(\Co{\oX}))}\), which is a tensor functor into 
  \(\Flat{E}\) (by Corollary~\ref{cor:proflatop}), and this process 
  determines an equivalence between the two kinds of tensor functors.  
  Further, an \(E_0\)-tensor functor determines, in an obvious manner, an 
  isomorphism \(\Co{u}:\Co{\sigma}\circ\fF\ra\fF\circ\Co{\tau}\) (in the 
  terminology of the definition).
  
  Hence, as discussed earlier in Remark~\ref{rmk:dualc}, we may instead work 
  with tensor functors from \(\cC\) to \(\Flat{E}\). Indeed, the dual 
  \(\Co{\tau}\) is used, for convenience, in the proof of 
  Theorem~\ref{thm:main} below, but everything can be translated back and 
  forth, by dualising. As in the concrete setup of the previous section, 
  objects of the form \(\tau(\oX)\) can be interpreted as ``functions'' on 
  \(\Co{\oX}\). See~\Sec{ssc:cschemes} for more details.
\end{remark}

\point\label{pt:estruct}
Corollary~\ref{cor:erep} and Proposition~\ref{prp:erep} show how, given an 
action of \(\FY_0\) on \(\sZ\) and an \(\FY_0\)-group \(\sG\), \(\Rep_\sG\) 
acquires an \(E_0\)-structure. We note that in this case, 
\(\Co{\tau}(\oX)=E\Ten_\mu\oX\) (by Lemma~\ref{lem:plngbase}).

Conversely, as mentioned above, given an \(E_0\)-tensor category \(\cC\) over 
\(\kk\), the functor \(\tau\) determines a map 
\(\tau_\1:\kk=\Eend(\1)\ra{}\Eend(\Co{E})=E\).  The two isomorphisms given with 
the \(E_0\)-structure on \(\cC\) show that this map corresponds to an action 
\(\mu:\spec(E)=\FY_0\x\sZ\ra\sZ\).

\begin{prop}
  The process described in~\ref{pt:estruct} determines a bijection between 
  actions of \(\FY_0\) on \(\sZ=\spec(\kk)\) and isomorphism classes of 
  \(E_0\)-structures on \(\Vec_\kk\) (all over \(\kk_0\))
\end{prop}
\begin{proof}
  This is a direct computation. Starting with an action \(\mu:\FY\ra\sZ\), 
  corresponding to a pro-algebra map \(f:\kk\ra{}E\), we have 
  \(\mu^*(\1)=\Co{E}\), and given an endomorphism \(a\in\kk\) of \(\1\), 
  \(\mu^*(a)\) is given by the ``right'' vector space structure on 
  \(\Co{E}\), via \(\mu\). Hence \(\tau_\1=f\).

  Conversely, since the functor \(\tau\) is exact, it is determined by its 
  value on \(\1\) (and \(\Eend(\1)\)), so by the map \(f=\tau_1:\kk\ra{}E\).
\end{proof}

\begin{defn}
  Let \(\cC\) be an \(E_0\)-tensor category. An \Def{\(E_0\)-fibre functor} 
  on \(\cC\) is an \(E_0\)-tensor functor from \(\cC\) to \(\Vec_\kk\), where 
  the latter has the \(E_0\)-structure corresponding to the action recovered 
  from \(\cC\).

  An \Def{\(E_0\)-Tannakian category} is an \(E_0\)-tensor category that 
  admits an \(E_0\)-fibre functor.

  More generally, if \(K\) is an \(\FY\)-field extension of \(\kk\), an 
  \Def{\(E_0\)-fibre functor over \(K\)} is an \(E_0\)-tensor functor from 
  \(\cC\) to \(\Vec_K\), with the corresponding \(E_0\)-structure.
\end{defn}

We may now formulate and prove the main Theorem: \(E_0\)-Tannakian categories 
are precisely categories of representations of (pro-) linear 
\(\FY_0\)-groups.

\begin{theorem}\label{thm:main}
  Let \(\w\) be an \(E_0\)-fibre functor on an \(E_0\)-tensor category 
  \(\cC\). Then there is a pro-linear \(E_0\)-group scheme \(\sG\) over 
  \(\kk\), and an action of \(\sG\) on each \(\w\), making \(\w\) an 
  \(E_0\)-tensor equivalence between \(\cC\) and \(\Rep_\sG\). If 
  \(\cC=\Rep_\sH\) for some pro-linear \(E_0\)-group scheme, then \(\sH\) is 
  canonically isomorphic to \(\sG\).
\end{theorem}
\begin{proof}
  Let \(\sG\) be the usual pro-linear group scheme \(\Aut^\Ten(\w)\) over 
  \(\kk\) associated to the fibre functor \(\w\). As indicated by 
  Proposition~\ref{prp:underly}, this should be the underlying pro-linear 
  group scheme, so our task is to give \(\sG\) the structure of a 
  \(\FY_0\)-scheme, over the \(\FY_0\)-structure on \(\sZ=\spec(\kk)\).  
  Thus, we should define an action map \(\mu:\FY_0\x\sG\ra\sG\) over \(\kk\), 
  where the domain is given the \(\kk\)-structure coming from the action map 
  \(\mu_0:\FY_0\x\sZ\ra\sZ\).
  
  In other words, we should provide a compatible system \(\mu_A\) of monoid 
  actions \(\mu_A:\FY_0(A)\x\sG(A)\ra\sG(A)\), one for each \(\kk_0\)-algebra 
  \(A\). Here, \(\sG(A)\) is the set of maps \(\spec(A)\ra\sG\) over 
  \(\kk_0\), and similarly for \(\FY_0\). Furthermore, these maps should 
  respect the \(\kk\)-structure in the following sense: Given an element 
  \(y\in\FY_0(A)\), and an element \(g\in\sG(A)\) mapping to an element 
  \(p\in\sZ(A)\), \(\mu_A(y,g)\) is a map of schemes \(h:\spec(A)\ra\sG\) 
  such that the diagram
  \begin{equation}\label{eq:twist}
    \xymatrix{
    \spec(A)\ar[r]^h\ar[d]_{(y,p)} & \sG\ar[d] \\
    \FY_0\x\sZ\ar[r]_{\mu_0} & \sZ
    }
  \end{equation}
  commutes. We denote by \(A^{(y)}\) the ring \(A\) with the \(\kk\)-algebra 
  structure coming from diagram~\eqref{eq:twist} (in the language 
  of~\ref{pt:plngadj}, \(\spec(A^{(y)})=\mu_!(\spec(A))\) as schemes over 
  \(\sZ\)).
  
  We thus fix a \(\kk\)-algebra \(A\).  By the definition of \(\sG\), we 
  should produce, for each \(A\)-point \(y\) of \(\FY_0\), and each tensor 
  automorphism \(g\) of \(A\Ten_\kk\w\) over \(A\), an automorphism 
  \(\mu(y,g)\) of \(A^{(y)}\Ten_\kk\w\), again over \(A\).

  An \(A\)-point of \(\FY_0\) factors, by definition, through a finite 
  sub-scheme \(\FY_0'\). Also by definition, the structure \(\tau\) restricts 
  to a tensor functor \(\Plng{E'}\), where \(\FY_0'=\spec(E_0')\), and 
  \(E'=E_0\Ten\kk\). So as long as the \(A\)-point \(y\) is fixed, we may 
  assume that \(E=E_0\Ten\kk\) is finite (the same is true for any finite 
  number of points).  We are thus given a map of \(\kk_0\)-algebras 
  \(y:E_0\ra{}A\) (in fact, we may at this point assume \(A=E\) and \(y\) the 
  identity, but this will be inconvenient when comparing several points).

  Consider now an object \(\oX\). The \(E\)-structure on the fibre functor 
  \(\w\) determines an \(E\)-module isomorphism 
  \(u_\oX:\w(\Co{\tau}(\oX))\ra{}E\Ten_\mu\w(\oX)\) (the right hand side is 
  what we denoted \(E^{(z)}\Ten_\kk\w(\oX)\), where \(z:E\ra{}E\) is the 
  identity). Using \(y\) we thus obtain an \(A\)-module isomorphism
  \begin{equation}
    A\Ten_E u_\oX:A\Ten_E\w(\Co{\tau}(\oX))\ra A^{(y)}\Ten\w(\oX)
  \end{equation}

  If \(g\in\sG(A)\) is an automorphism of \(A\Ten_\kk\w\), we note that   
  \(g_{\Co{\tau}(\oX)}:A\Ten_\kk\w(\Co{\tau}(\oX))\ra{}A\Ten_\kk\w(\Co{\tau}(\oX))\) 
  is an \(E\)-module automorphism (since \(E\) acts on objects in \(\cC\)), 
  so it descends to an automorphism of \(A\Ten_E\w(\Co{\tau}(\oX))\).
  Hence we obtain an induced map
  \begin{equation}\label{eq:act}
    \xymatrix{
    A\Ten_E\w(\Co{\tau}(\oX))\ar[r]^{A\Ten_E u_\oX}
    \ar[d]_{g_{\Co{\tau}(\oX)}} & 
    A^{(y)}\Ten_\kk\w(\oX)\ar[d]^{\mu(y,g)_\oX}\\
    A\Ten_E\w(\Co{\tau}(\oX))\ar[r]^{A\Ten_E u_\oX} & A^{(y)}\Ten_\kk\w(\oX)
    }
  \end{equation}
  with \(\mu(y,g)\) the map as indicated. This concludes the definition of 
  \(\mu\).

  To verify that the map is a monoid action, if, in the above definition, 
  \(y:E_0\ra\kk\) corresponds to the identity of \(\FY_0\), then 
  \(A^{(y)}\Ten_\kk\w(\oX)=\w(\oX)\), and the induced map in~\eqref{eq:act} 
  is just \(g_\oX\) (using the existence of the isomorphism \(a\) 
  from~\eqref{eq:etensor} of Definition~\ref{def:etensor}).

  Next, assume that we are given two points \(y_1,y_2:E_0\ra{}A\), and let 
  \(y=(y_1,y_2)\circ{}m^\#\) be their product in \(\FY_0\).  We first note 
  that from the fact that we have an action on \(\sZ\) we get an isomorphism 
  (over \(\kk\)) \({(A^{(y_2)})}^{(y_1)}\ra{}A^{(y)}\). Now, applying the 
  definition~\eqref{eq:act} twice, we get a diagram
  \begin{equation}
    \xymatrix{
    A^{(y_2)}\Ten_{y_1}\w(\Co{\tau}\Co{\tau}(\oX))\ar[r]
    \ar[d]_{g_{\Co{\tau}\Co{\tau}(\oX)}} & 
    {A^{(y_2)}}^{(y_1)}\Ten_\kk\w(\oX)\ar[d]^{\mu(y_1,\mu(y_2,g))_\oX}\\
    A^{(y_2)}\Ten_{y_1}\w(\Co{\tau}\Co{\tau}(\oX))\ar[r] & 
    {A^{(y_2)}}^{(y_1)}\Ten_\kk\w(\oX)
    }
  \end{equation}

  Applying the isomorphism \(b\) from Definition~\ref{def:etensor}, we may 
  replace \(\w(\Co{\tau}\Co{\tau}(\oX))\) with 
  \(E\Ten_\kk{}E\Ten_m\w(\Co{\tau}(\oX))\), so the left part of the diagram 
  becomes \(A\Ten_E\w(\Co{\tau}(\oX))\), with the \(E\)-structure on \(A\) 
  given by \(y\). This concludes the proof that \(\mu\) is an action. The fact 
  that \(\w\) induces an equivalence of \(\cC\) with \(\Rep_\sG\) follows from 
  the adjunction, as in Proposition~\ref{prp:underly}.

  For the last part, assume that \(\cC\) is the \(E_0\)-tensor category 
  associated with a pro-linear \(\FY\)-group scheme. By the usual Tannakian 
  formalism and the construction, the underlying group scheme is \(\sG\). Thus, 
  we need only verify that the action of \(\FY\) is the same. This is clear, 
  since the functor \(\tau\) was defined in the same way for representations of 
  \(\sG\) and for vector spaces.
\end{proof}

\begin{remark}
  Analogously to the algebraic case, if \(\sT=\spec(B)\) is an \(\FY\)-scheme 
  over \(\sZ\), and \(g_0:\sT\ra\sG\) is a \(\sT\)-point of \(\sG\) (so it 
  commutes with the action), then \(g\) determines an automorphism of 
  \(B\Ten_\kk\w\) as an \emph{\(E_0\)-functor}. Indeed, such a point 
  determines an automorphism \(g_0\) of \(B\Ten_\kk\w\) as a tensor functor.  
  Let \(A=B\Ten_\kk{}E\), let \(y:E\ra{}A\) be the obvious map, and let \(g\) 
  be the induced automorphism of \(A\Ten\w\). Then, by diagram~\eqref{eq:act} 
  we have a diagram
  \begin{equation}
    \xymatrix{
    B\Ten_\kk\w(\Co{\tau}(\oX))\ar[r]^{A\Ten_E u_\oX}
    \ar[d]_{g_{\Co{\tau}(\oX)}} & 
    A^{(y)}\Ten_\kk\w(\oX)\ar[d]^{\mu(y,g)_\oX}\\
    B\Ten_\kk\w(\Co{\tau}(\oX))\ar[r]^{A\Ten_E u_\oX} & 
    A^{(y)}\Ten_\kk\w(\oX)
    }
  \end{equation}
  Now, since the action of \(y\) on \(\sG\) commutes with that on \(\sT\), 
  the map on the right is \((B\Ten{}E_0)\Ten_B{g_0}_\oX\) (tensor product 
  with respect to the action \(B\ra{}B\Ten{}E_0\)). Composing with the 
  isomorphism 
  \(B\Ten{}E_0\Ten_B{}B\Ten_\kk\w(\oX)\ra{}B\Ten_\kk{}E\Ten_\mu\w(\oX)\) 
  (coming from the fact that that action on \(\sT\) is over \(\sZ\)), we get 
  the compatibility required in Definition~\ref{def:etensor}.
\end{remark}

\subsection{\(\FY\)-schemes in \(\cC\)}\label{ssc:cschemes}
As with usual tensor categories, it is possible to define the category of 
(affine) \(\FY\)-schemes. This is the opposite category to the category of 
\(\FY\)-algebras in \(\cC\), defined as follows. Recall, first, that an 
algebra in \(\cC\) is an ind-object \(\oX\) of \(\cC\), together with  
(suitable) maps \(m:\oX\Ten\oX\ra\oX\) and \(u:\1\ra\oX\). If \(\cC\) is 
given with an \(E_0\)-structure \((\tau,a,b)\), the tensor structure on 
\(\tau\) makes \(\tau(\oX)\) an algebra as well. We note that the isomorphism 
\(a\) induces a map \(a_\oX:\oX\ra\tau\oX\), which is an algebra map (since 
\(a\) is a tensor isomorphism).

We recall that in the case of usual algebras, \(\tau(\oX)\) was the analogue 
of the algebra of functions on the arc space of the scheme associated to 
\(\oX\). Hence the following definition is natural.

\begin{defn}\label{def:escheme}
  Let \((\cC,\tau,a,b)\) be an \(E\)-tensor category, \(\oX\) an algebra in 
  \(\cC\).  An \Def{\(E\)-structure} on \(\oX\) consists of an algebra map 
  \(d:\tau(\oX)\ra\oX\), such that \(d\circ{}a_\oX\) is the identity, and the 
  following diagram commutes
  \begin{equation}\label{eq:ealg}
    \xymatrix{
    \tau\tau\oX\ar[rr]^{\tau(d)}\ar[d]_{b_\oX}&&\tau\oX\ar[d]^d\\
    m^!\tau\oX\ar[r]_m & \tau\oX\ar[r]_d & \oX
    }
  \end{equation}

  The category of (affine) \Def{\(\FY_0\)-schemes in \(\cC\)} is the opposite 
  of the category of algebras with \(E\)-structure (and maps of between them 
  that preserve this structure)
\end{defn}

As in previously known cases, any object \(\oX\) has an associated ``affine 
space'' \(\cC\)-scheme \(A(\oX)\). These affine spaces were important in the 
differential case to achieve elimination of imaginaries in the corresponding 
theory (cf~\Cite[4.4--4.5]{tannakian}), and played an important role 
in~\Cite{Deligne}.

The construction of \(A(\oX)\) is a special case of the following result, 
which says that prolongation spaces also exist in \(\cC\).  The defining 
property is taken to be analogous to the adjunction in~\ref{pt:prolong} (the 
affine case should also be compared to Proposition~\ref{prp:plngaff}).

\begin{prop}
  Let \(\cC\) by an \(E\)-tensor category. The forgetful functor from 
  \(\FY\)-schemes to schemes in \(\cC\) has a right adjoint.
\end{prop}
\begin{proof}
  Given a (usual) scheme with algebra \(\oW\) in \(\cC\), the induced  
  \(\FY\)-scheme corresponds to \(\tau\oW\), with \(d:\tau\tau\oW\ra\tau\oW\) 
  induced by the product on \(\FY\). The fact that this is an \(\FY\)-scheme 
  structure comes from the monoid axioms on \(\FY\).
  
  Given some other \(\FY\)-scheme \(\oX\), \(d:\tau\oX\ra\oX\), an algebra 
  map \(f:\oW\ra\oX\) determines a algebra maps \(\tau(f):\tau\oW\ra\tau\oX\) 
  and \(\tau^2(f):\tau^2\oW\ra\tau^2\oX\), resulting, upon composition with 
  \(d\) and \(\tau(d)\), in maps \(\ti{f}:\tau\oW\ra\oX\) and 
  \(\ti{f}_\tau:\tau\tau\oW\ra\tau\oX\). The fact that these maps determine a 
  map of \(\FY\)-schemes comes from diagram~\eqref{eq:ealg}.
  
  In the other direction, a map \(\tau\oW\ra\oX\) restricts to a map 
  \(\oW\ra\oX\). It is standard to check that these are inverse to each 
  other.
\end{proof}

\begin{remark}
  If \(\FY\) is the free formal monoid generated by some formal set 
  \(\FY_1\), the data in Definition~\ref{def:escheme} can be given in terms 
  of \(\FY_1\), as a map \(d_1:\tau_1(\oX)\ra\oX\), without the 
  condition~\eqref{eq:ealg}. The full \(d:\tau\oX\ra\oX\) can then be 
  reconstructed as the unique map extending \(d_1\) and 
  satisfying~\eqref{eq:ealg}. In the case where 
  \(\FY_1=\spec(\kk_0[\epsilon])\) (where \(\kk_0\) is a field of 
  characteristic \(0\) and \(\epsilon^2=0\)), the corresponding \(\FY\) is 
  the additive formal group (Example~\ref{ex:freeadd}), and the explicit 
  construction of \(\tau\oX\) and of \(d\) was (essentially) carried out 
  in~\Cite[4.4.3]{tannakian} (\(\tau\oX\) is the ind-object with maps \(q_n\) 
  there; the map \(d\) is essentially given by the \(t_n\)).
\end{remark}

\subsection{Passing to a sub-field}\label{ssc:descent}
In this section we prove an analogue of~\Deligne{Corollary~6.20}. The 
corollary says that if a Tannakian category \(\cC\) over a field \(\kk\) 
admits a fibre functor over some \(\kk\)-algebra, then it admits a fibre 
functor over a finite extension of \(\kk\). Put another way, a Tannakian 
category over a field \(\kk\) admits a fibre functor over the algebraic 
closure of \(\kk\). This result is useful, for example, in applications to 
Galois theory, as presented in the Introduction~(\Sec{ssc:galois}).

We maintain the notation and assumptions of~\Sec{ssc:lingp}. For convenience, 
we assume that \(\FY\) is a (formal) group.

\begin{defn}
  Let \(\kappa\) be a cardinal. We say that the formal group \(\FY\) is 
  \Def{\(\LK\)-generated} if there is a formal set \(\FZ\), such that \(\FY\) 
  is a quotient of the free formal group generated by \(\FZ\), and \(\FZ\) 
  can be presented by a filtering system of size smaller than \(\kappa\).

  We say that \(\FY\) is \Def{\(\LK\)-presented} if it is 
  \(\LK\)-generated, and the kernel of the quotient map is also 
  \(\LK\)-generated.

  We say finitely or countably generated (or presented) if \(\kappa\) is 
  \(\w\) or \(\w_1\), respectively.
\end{defn}

For example, a (discrete) group is finitely generated or presented if it is 
such in the usual sense. We note that when \(\kappa>\w\), we may replace this 
by the condition that \(\FY\) itself is given by a system of cardinality less 
than \(\kappa\).

\begin{example}
  If \(\kk\) has characteristic \(0\), then \(\spec(\kk[[x]])\) is finitely 
  generated (and presented) by \(\spec(\kk[\eps])\) (according to 
  Example~\ref{ex:freeadd}). In characteristic \(p>0\), it is countably 
  presented, since the quotients to finitely many of the variables \(e_i\) 
  (of the algebra \(TE\) in the same example) are all finite.
\end{example}

From now on, we fix \(\kappa\), and assume that \(\FY\) is 
\(\LK\)-presented.

\begin{defn}
  We say that an \(\FY\)-field \(\kk\) is \Def{\(\FY\)-closed} if any 
  non-empty \(\FY\)-pro-variety over \(\kk\) that can be given by a 
  \(\LK\)-system of varieties, has a \(\kk\)-point (i.e., an 
  \(\FY\)-morphism from \(\spec(\kk)\)).
\end{defn}

Model theoretically, an \(\FY\)-closed field is a weak version of a universal 
domain. We note that, unlike in algebraic geometry, a non empty 
\(\FY\)-variety need not have a point in an \emph{\(\FY\)-field}. For 
instance, the system of equations \(x^2=1,\sigma(x)=-x\) (where \(\sigma\) is 
an automorphism), has no solution in a field. In this case, \(\FY\)-closed 
fields do not exist, and the statement of Proposition~\ref{prp:descent} below  
is empty.  The proof does show, however, that there is a fibre functor over 
an \(\FY\)-algebra over \(\kk\), which is finite generated and ``simple'' as 
an \(\FY\)-algebra.

We note that when \(\FY\) is local, the action of it on a variety restricts 
to an action on each irreducible component, and therefore to its generic 
point, a field. Hence, in this case, \(\FY\)-closed fields exist by the 
standard arguments.

We say that an \(\FY\)-tensor category is generated by a collection \(S\) of 
objects if \(S_\tau=\St{\tau_E(\oX)}{\oX\in{}S}\) generates it as a tensor 
category. We note that, since \(\tau\) is an exact tensor functor, the tensor 
category generated by \(S_\tau\) is automatically an \(\FY\)-tensor category.

The following proposition is an analogue of~\Deligne{6.20} for our setting.  
As before, we only need to reduce to it, rather than re-prove it, using the 
current framework.

\begin{prop}\label{prp:descent}
  Let \((\cC,\tau,a,b)\) be an \(\FY\)-tensor category over an \(\FY\)-closed 
  field \(\kk\). Assume that \(\cC\) is generated as an \(\FY\)-tensor 
  category by one object, and that it admits a fibre functor over some 
  \(\FY\)-field extension of \(\kk\). Then it admit a fibre functor over 
  \(\kk\).
\end{prop}
\begin{proof}
  Let \(\FY=\spec(E)\) be finite.  We first note that in the case of vector 
  bundles, the operation \(\Co{\tau}\) can be described geometrically as 
  follows. Let \(\sX\) be a scheme over \(\sZ=\spec(\kk)\). By definition of 
  \(\mu^*\), we have a map \(\mu^*(\sX)\ra\sX\), and adjunction provides a 
  map \(p^*p_*\mu^*(\sX)\ra\mu^*(\sX)\). Let \(r\) be the composed map. Again 
  by the definition of \(p^*\) we also have a map 
  \(s:p^*p_*\mu^*(\sX)\ra{}p_*(\mu^*(\sX))=\tau(\sX)\). Now, if \(V\) is a 
  vector bundle on \(\sX\), we obtain a vector bundle \(s_*r^*(V)\) on 
  \(\tau(\sX)\) (this is a vector bundle since \(s\) is finite and flat). In 
  the case when \(\sX=\sZ\), this is \(\Co{\tau}\).

  Now, assume that \(\cC\) is generated (as an \(\FY\)-tensor category) by an 
  object \(\oX\), and has a fibre functor over some \(\FY\)-scheme over 
  \(\kk\).  Let \(\cC_0\) be the tensor category generated by \(\oX\).  
  According to the proof of~\Deligne{6.20}, there is an affine variety 
  \(\sS_0\) over \(\kk\), a fibre functor \(\w_0\) of \(\cC_0\) over 
  \(\sS_0\), and a faithfully flat groupoid scheme \(\sG_0\) over \(\sS_0\) 
  such that \(\cC_0\) is identified by \(\w_0\) with the category of 
  representations of \(\sG_0\).

  Similarly, the tensor category \(\cC_E\) generated by \(\oX\) and 
  \(\Co{\tau}(\oX)\) is equivalent the category of representations of a 
  faithfully flat groupoid scheme \(\sG_E\) over a variety \(\sS_E\). By the 
  first paragraph, the \(\FY\)-structure on \(\w\) produces (perhaps after a 
  flat, finite type localisation) a map \(\sS_E\ra\tau\sS_0\).

  Now, dropping the assumption that \(\FY\) is finite, we iterate the 
  construction for a system of size \(\LK\). We obtain a projective 
  system \(\sS=(\sS_E)\) of varieties, and a fibre functor \(\bar{\w}\) over 
  \(\sS\). The system \(\sS\) inherits an \(\FY\)-structure from the 
  prolongations, and by construction, \(\bar{\w}\) is an \(\FY\)-tensor 
  functor. By the assumption on \(\kk\), \(\sS\) has a \(\kk\)-point \(s\).  
  The fibre \(\bar{w}_s\) is an \(\FY\)-fibre functor over \(\kk\).
\end{proof}

\section{Questions and Speculations}\label{sec:questions}
In this section I point out some questions and other issues I would like to 
clarify. At least some of them should be easy to answer, but I do not see the 
answer immediately, and they are not directly relevant to the main point of 
the paper, so I leave them unanswered. Nevertheless, I think they are 
interesting.

\subsection{Sheaves on formal sets}\label{q:ablim}
As explained in~\Sec{ssc:fibred}, if \(\FY\) is a formal set (viewed as a 
filtering system), and \(\cC\ra\FY\) is a fibred category over \(\FY\), whose 
fibres are (say) sheaves of some kind over the corresponding finite piece, 
then \(\Lim[\FY]{\cC}\) can be viewed as the category of sheaves of the same 
kind on \(\FY\). However, it does not seem to be straightforward to deduce 
properties of \(\Lim[\FY]{\cC}\) from properties of the fibres.

For instance, assume that each fibre is abelian. May we conclude that the 
limit is abelian?  The answer is ``no'' in general, and ``yes'' if each 
pullback functor is exact. However, in our situation this assumption does not 
hold. In the context of~\ref{pt:pullbacks}, we know that the pullbacks are 
either left or right-exact (and, indeed, admit a left or right adjoint), but 
not both. Can anything be said in this case? In special cases such as 
completions of Noetherian local rings and finitely generated modules, one 
ends up with an abelian category, but this does depend on the Artin--Rees 
Lemma or similar results.

More generally, this appears like it should be a classical construction, but 
I don't know what would be a good reference.

\subsection{Cartier duality}
The usual Tannakian formalism can be viewed as a generalisation of Cartier 
duality to more general groups: Given a rigid tensor category \(\cC\), the 
tensor product determines a group structure on the set of isomorphism classes 
of invertible objects (in other words, this is the group of invertible 
elements in the Grothendieck ring of \(\cC\)). We may call this group the 
Picard group of \(\cC\). When \(\cC\) is the category of representations of 
an algebraic group \(\sG\) of multiplicative type (say, \(\sG_m\)), it is 
determined by the invertible objects, and the Picard group of \(\cC\) is the 
Cartier dual of \(\sG\). For general \(\sG\), it is thus reasonable to view 
\(\cC=\Rep_\sG\) as an analogue of the Cartier dual (and the recovery of 
\(\sG\) from \(\cC\) an analogue of recovering \(\sG\) from its dual).


When \(\sG=\sG_a\), there are no non-trivial invertible representations, so 
the usual Cartier dual carries no information, but in 
Example~\ref{ex:cartierga} it is shown that (in characteristic \(0\)), the 
additive formal group should be viewed as the Cartier dual of \(\sG_a\), 
which suggests that this formal group is in some sense the Picard group for 
the category of representations of \(\sG_a\).  The question is how to recover 
this formal group directly from the category \(\cC=\Rep_{\sG_a}\), and more 
generally, whether one can compute a meaningful (formal) Picard group like 
that for an arbitrary tensor category.

Another question related to the duality: In characteristic \(0\), the group 
schemes \(\sG_m\) and \(\sG_a\) correspond, respectively, to the cases of an 
automorphisms and a derivation, and they are the only affine groups of 
dimension \(1\). So it seems that we have shown that the only ``rank \(1\)'' 
operators in characteristic \(0\) are automorphisms and derivations. The 
question is how to explain what ``rank \(1\)'' means, without going through 
Cartier duality (this might be related to the classification mentioned 
in~\Cite[2.4]{buium}).

\subsection{Quotients}
In the usual treatment of differential and difference fields, an important 
role is played by the ``field of constants''. It played no role in this 
paper, but it is still interesting to define it in the general context 
of~\Sec{sec:ms}.

We have a categorical description. Given a scheme \(\oX\) (over \(\kk_0\)), 
one may view \(\oX\) as an \(\FY\)-scheme \(\ul{\oX}\) via the trivial 
action. We may then define quotient by \(\FY\) to be the ``left-adjoint'' to 
this functor: Given an \(\FY\)-scheme \(\oZ\), \(\oZ/\FY\) is defined as a 
\emph{covariant} functor on schemes \(\oX\) by 
\((\oZ/\FY)(\oX)=\Ehom_\FY(\oZ,\ul{\oX})\). The problem is that there is no 
reason that this functor should be representable (unless, of course, \(\FY\) 
is finite), and furthermore, it seems impossible to describe maps from a 
scheme to \(\oZ/\FY\).

In the affine case, we do have an algebra associated to \(\oZ/\FY\): if 
\(\oZ=\spec(\kk)\), with action of \(\FY\) given by \(m\) and projection 
given by \(p\) (both on the level of algebras), then \(\oZ/\FY\) corresponds 
to the sub-algebra given by \(p(a)=m(a)\). If \(\kk\) is a field, then this 
sub-algebra is a field as well, and this definition coincides with the usual 
one in the difference and differential case.

\subsection{Relation to crystals}
We note that the definition of an \(E\)-structure on an object \(\oX\), given 
in~\ref{def:escheme} for algebra objects \(\oX\), makes sense also for 
objects \(\oX\) of \(\cC\) itself (of course, \(d\) is no longer an algebra 
map).

Let \(\sY\) be a variety over a field \(\kk_0\) of characteristic \(0\).  
Recall (\Cite{lurie} or~\Cite[\S~1.5]{crystal}) that a \Def{crystal of 
quasi-coherent sheaves} on \(\sY\) consists of a quasi-coherent sheaf \(\qF\) 
on \(\sY\), together with isomorphisms \(\eta_{x,y}:\qF_x\ra\qF_y\) for any 
infinitesimally close \(R\)-points \(x\) and \(y\) of \(\sY\) (i.e., \(x\) 
and \(y\) have the same restriction to \(\spec(R)_{red}\); \(R\) is any 
\(\kk_0\)-algebra).  These isomorphisms should satisfy some natural 
properties (so that \(\qF\) is a ``locally trivial'' sheaf on the 
infinitesimal site of \(\sY\)).  We note that when the projection from \(R\) 
to \(R/I\) (\(I\) the nilradical of \(R\)) has a section, it is enough to 
specify these isomorphisms when \(x\) is the restriction of \(y\) to 
\(\spec(R/I)\).

Let \(\qF\) be such a crystal, and let \(\kk\) be the field of rational 
functions on \(\sY\). A derivation of \(\kk\) over \(\kk_0\) (i.e., a 
meromorphic vector-field on \(\sY\)) determines, for each \(n\), an 
\(E_n\)-point \(y\) of \(\sY\), where \(E_n=\kk[x]/x^{n+1}\), and the crystal 
data provides an isomorphism \(\qF_x\ra\qF_y\), where \(x\) is the point 
corresponding to the \(0\) vector-field. This is the same as an 
\(E_n\)-isomorphism \(E_n\Ten_y{}M\ra{}E_n\Ten{}M\), where \(M\) is the fibre 
of \(\qF\) on the generic point. Composing \(y\) with the co-multiplication 
\(m\) of \(E\) we likewise get isomorphisms involving \(E\Ten{}E\), and the 
compatibility conditions on the crystal imply that the 
diagram~\eqref{eq:ealg} commutes. Hence, a crystal structure on a 
quasi-coherent sheaf determines, for each vector-field, an 
\(E=\kk_0[[x]]\)-structure on it. A similar analysis applies crystals in 
other categories (e.g., a crystal of schemes, as in~\Cite{lurie}), and also 
for the Crystalline site in positive characteristic 
(cf.~\Cite[Prop.~5.1]{crystal}; note that the Crystalline site corresponds to 
\emph{usual}, rather than Hasse--Schmidt derivations, as in 
Example~\ref{ex:freeadd}). It would be interesting to understand the precise 
relation, and whether it is useful.


\subsection{Changing the monoid}
Throughout, we work with a fixed formal monoid \(\FY\). It makes sense, of 
course, to ask what happens when we let \(\FY\) vary. For example, in the 
context of several derivations, it could be desirable to pass to a subset of 
the derivations, or to a more convenient choice of them.

In particular, in the context of the Tannakian formalism, \(\FY\) is 
recovered from \(\Plng{E}\) (as \(\Eend(\1_E)\)), so one could ask to replace 
\(\Plng{E}\) by an abstract category of prolongations \(\cD\). This would 
entail finding conditions under which \(\cD\) is canonically isomorphic to 
\(\Plng{E}\) for \(E=\Eend(\1_\cD)\) (over a given tensor functor 
\(\cC\ra\cD\)). Such a formalism would treat all formal monoid actions at 
once. I leave it to some other time.

\subsection{More general monads}\label{q:monads}
Instead of working with with a formal monoid \(\FY\), as we did, we could 
work more generally with the corresponding monad \(\fW_\FY\), given by 
\(\fW_\FY(\sZ)=\FY\x\sZ\). The advantage is that we may then forget about 
\(\FY\) and work just with a monad \(\fW\) on the category of ind-schemes 
(over a given base \(\kk_0\); we would probably be assuming that \(\fW\) is 
``continuous'', i.e., determined by its restriction to schemes). Given such a 
monad, an \(\FY\)-scheme is replaced by a \(\fW\)-algebra \(\sZ\), and 
likewise \(\FY\)-schemes over \(\sZ\) are replaced by \(\fW\)-algebras over 
\(\sZ\). The main difference with our approach is that we are no longer 
assuming to have a functorial map \(p:\fW(\oX)\ra\oX\) (the projection).  
There are at least two interesting examples covered only by this more general 
approach: The \(p\)-adic Witt scheme of length \(2\), \(\fW=\fW_2\), 
corresponding to the arithmetic differential equations of~\Cite{buium} (See 
especially~\Cite[\S~2.4]{buium}), and its global analogue, the big Witt 
vector functor, corresponding to the theory \(\Lambda\)-spaces 
of~\Cite{borger} (which are offered there as a notion of spaces over 
\(\FF_1\)).

If \(\fW\) happens to have a right adjoint \(\tau_0\) (possibly going from 
schemes to pro-schemes), which is then automatically a co-monad, then an 
algebra \(\sZ\) for \(\fW\) determines a co-algebra \(t:\sZ\ra\tau_0\sZ\) for 
\(\tau_0\). Given a scheme \(\sX\) over \(\sZ\), we set 
\(\tau(\sX)=\tau_0(\sX)\x_t\sZ\), and call it the prolongation of \(\sX\) 
(viewed as a pro-scheme over \(\sZ\)). As in~\ref{pt:prolong}, \(\tau\) is a 
co-monad on pro-schemes over \(\sZ\), and is, by construction, right adjoint 
to the forgetful functor from \(\fW\)-schemes to schemes.

The main issue with extending the results of the paper is now to find the 
analogue of tensoring with \(E\) to this setting, i.e., we need a canonical 
way to extend \(\fW\) (or \(\tau\)) to a tensor category over \(\sZ\). This 
should be possible.

\printbibliography
\end{document}